\documentclass[12pt]{amsart}
\usepackage{amsmath,amsfonts,amssymb}
\usepackage{color,graphicx}
\usepackage[left=2.5cm,right=2.5cm,top=3.5cm,bottom=3cm]{geometry}

\usepackage{amsthm}
\usepackage{ dsfont }
\usepackage{bm}
\usepackage{mathtools}
\usepackage{enumerate}

\newtheorem{theorem}{Theorem}[section]

\newtheorem{taggedtheoremx}{Theorem}
\newenvironment{taggedtheorem}[1]
{\renewcommand\thetaggedtheoremx{#1}\taggedtheoremx}
{\endtaggedtheoremx}
\newtheorem{lemma}[theorem]{Lemma}

\newtheorem{corollary}[theorem]{Corollary}
\newtheorem{proposition}[theorem]{Proposition}

\newtheorem{example}[theorem]{Example}
\newtheorem{remark}[theorem]{Remark}
\usepackage[utf8]{inputenc}

\newcommand{\F}{\mathcal{F}}
\newcommand{\HH}{\mathcal{H}}
\newcommand{\PP}{\mathcal{P}}
\newcommand{\II}{\mathcal{I}}
\newcommand{\QQ}{\mathcal{Q}}
\newcommand{\N}{\mathds{N}}
\newcommand{\Z}{\mathds{Z}}
\newcommand{\C}{\mathds{C}}
\newcommand{\R}{\mathds{R}}
\newcommand{\D}{\mathds{D}}
\newcommand{\Q}{\mathds{Q}}
\newcommand{\T}{\mathds{T}}

\newcommand{\mcn}{\mathcal N}

\date{\today}
\title{Cyclicity of composition operators on the Fock Space}

\author{Frédéric Bayart and Sebastián Tapia-García}
\address{Frédéric Bayart, Sebastián Tapia-García}
\address{Laboratoire de Math\'ematiques Blaise Pascal UMR 6620 CNRS, Universit\'e Clermont Auvergne, Campus universitaire C\'ezeaux, 3 place Vasarely, 63178 Aubière Cedex, France.}
\address{Email address: frederic.bayart@uca.fr, sgtg22@gmail.com}

\begin{document}

\maketitle

\begin{abstract}In this paper we provide a full characterization of cyclic composition operators defined on the $d$-dimensional Fock space $\F(\C^d)$ in terms of their symbol. Also, we study the supercyclicity and convex-cyclicity of this type of operators. We end this work by computing the approximation numbers of compact composition operators defined on $\F(\C^d)$.
\end{abstract}

\section{Introduction}
Let $\C^d$ be the $d$-dimensional complex Euclidean space with $d\geq 1$. 
The classical Fock space on $\C^d$ is defined by
\[\F(\C^d):=\left\{f\in \mathcal{H}(\C^d):~\|f\|^2:=\dfrac{1}{(2\pi)^d}\int_{\C^d}|f(z)|^2e^{-\frac{|z|^2}{2}}dA(z)<\infty\right \},\]
where $dA$ stands for the Lebesgue measure on $\C^d$, $|z|=\sqrt{\langle z,z\rangle}$ and $\langle z,w\rangle=\sum_{j=1}^d z_j \overline{w_j}$. 
The space $\mathcal F(\C^d)$ is a separable Hilbert space equipped with the inner product
\[\langle f,g \rangle:= \dfrac{1}{(2\pi)^d}\int_{\C^d} f(z)\overline{g(z)}e^{-\frac{|z|^2}{2}}dA(z). \]
Observe that we do not distinguish between the inner products of $\C^d$ and $\F(\C^d)$.
The space $\F(\C^d)$ is also a reproducing kernel Hilbert space. 
It is well-know that the reproducing kernel function is given by 
\[k_w:~z\in\C^d\mapsto k_w(z)= \exp\left(\frac{\langle z,w\rangle }{2}\right),\] 
which has norm $\|k_w\|=\exp\left(\frac{|w|^2}{4}\right)$.
Also, the set of polynomials $\{z^\alpha:\alpha\in\N^d\}$ forms an orthogonal basis of $\F(\C^d)$ and \[\|z^\alpha\|^2= 2^{|\alpha|}\prod_{j=1}^d\alpha_j!.\]
In this paper, we are interested in composition operators on $\mathcal F(\C^d)$. 
Let $\varphi:\C^d\to\C^d$ be a holomorphic function. The composition operator with symbol $\varphi$ is defined by
\[f\in \mathcal{H}(\C^d)\mapsto C_\varphi(f):= f\circ \varphi.\] 

Carswell et al \cite{CMS} have characterized when $C_\varphi$ defines a bounded composition
operator on $\mathcal F(\C^d)$: this holds if and only if $\varphi(z)=Az+b$ where $A\in \C^{d\times d}$, with $\|A\|\leq 1$ and $b\in\C^d$ which satisfy $\langle Av,b\rangle=0$ for all $v\in\C^d$ with $|Av|=|v|$. 
Moreover, $C_\varphi$ is compact if and only if $\|A\|<1$.
Since then, many works have been done to characterize properties of $C_\varphi$
in terms of the properties of the symbol $\varphi$, sometimes only when $d=1$:
see for instance \cite{Du,FZ,GI,JPZ}.
%

In this paper we are concerned in the dynamical properties of composition operators defined on $\F(\C^d)$.  Let us recall the relevant definitions. Let $X$ be a separable Banach space, let $T\in\mathcal{L}(X)$ be a bounded linear operator defined on $X$ and let $x\in X$. 
The orbit of $x$ under the action of $T$ is the set $\textup{orb}(T,x):=\{T^nx:~n\in\N\}$. 
The operator $T\in \mathcal{L}(X)$ is said cyclic if there is $x\in X$ such that $\textup{span}(\textup{orb}(T,x))$ is dense in $X$. 
In this case, we say that $x$ is a cyclic vector for $T$.
Similarly, we say that $T$ is supercyclic or hypercyclic if there is $x\in X$ such that $\C\cdot \textup{orb}(T,x)$ or $\textup{orb}(T,x)$ is dense in $X$ respectively.

There is a rich literature concerning cyclicity, supercyclicity or hypercyclicity of composition operators defined on the Hardy space, Bergman space, Dirichlet space; see for instance \cite{BM,BS,DK,ZZ}.\\

Regarding Fock spaces, the cyclic composition operators on the Fock space of $\C$ have been characterized in \cite[Theorem 4.2]{GI}.

\begin{taggedtheorem}{A}\label{Theorem A}
	Let $a,b\in \C$ be such that $C_{az+b}$ induces a bounded composition operator on $\F(\C)$.
	Then, $C_{az+b}$ is cyclic on $\mathcal{F}(\C)$ if and only 
$a\neq 0$ and $a$ is not a root of unity. 
\end{taggedtheorem}

In higher dimensions, cyclicity has only been characterized in the very particular case where 
$A$ is diagonal and unitary (note that this implies $b=0$). 

\begin{taggedtheorem}{B}\cite[Theorem 5.3]{JPZ}\label{Theorem B}
	Let $A=\textup{diag}(e^{i\theta_1},...,e^{i\theta_d})\in \C^{d\times d}$.
	The composition operator $C_{Az}$ is cyclic on $\mathcal{F}(\C^d)$ if and only if the set $\{\pi,\theta_1,...,\theta_n\}$ is $\Q$-linearly independent.
\end{taggedtheorem}

Our main result, which extends both Theorem \ref{Theorem A} and Theorem \ref{Theorem B}, solves the problem of cyclicity in full generality. 

\begin{theorem}\label{Theorem 1}
	Let $A\in\C^{d\times d}$, $b\in \C^d$ and $\varphi(z)=Az+b$ be such that $C_\varphi$ induces a bounded composition operator on $\mathcal{F}(\C^d)$.
	Then, $C_\varphi$ is cyclic if and only if the following conditions are satisfied :
	\begin{itemize}
	\item $A$ is invertible;
	\item $A$ is diagonalizable or its canonical Jordan form admits at most one Jordan block, whose size is exactly $2$;
	\item  if $\lambda:=(\lambda_j)_{j=1}^{\hat{d}}$ denotes the sequence of eigenvalues of $A$, repeated by geometric multiplicity (therefore $\hat{d}= d$ or $d-1$), then for any $\alpha\in \Z^{\hat{d}}\setminus \{0\}^{\hat{d}}$, $\lambda^\alpha\neq 1$.
	\end{itemize}
\end{theorem}

Observe that the third condition may be rewritten by saying that
 if $\lambda:=(\lambda_j)_{j=1}^{\hat{d}}$ denotes the sequence of eigenvalues of $A$, repeated by geometric multiplicity (therefore $\hat{d}= d$ or $d-1$), then
	for any $\alpha\in \Z^{\hat{d}}\setminus \{0\}^{\hat{d}}$ such that $\lambda^\alpha=\exp(i\theta)$ for some $\theta\in\R$, the pair $\{\pi,\theta\}$ is $\Q$-linearly independent. 
	
\begin{example}
\begin{itemize}
\item Let $A=\begin{pmatrix} \frac{e^{i\theta_1}}2&0\\0&\frac{e^{i\theta_2}}3\end{pmatrix}$. Then $C_{Az}$ is cyclic on $\mathcal F(\C^2)$ for all values of $\theta_1,\theta_2$.
\item Let $A=\begin{pmatrix} \frac{e^{i\theta_1}}2&0\\0&\frac{e^{i\theta_2}}4\end{pmatrix}$. Then $C_{Az}$ is cyclic on $\mathcal F(\C^2)$ if and only if $\theta_2-2\theta_1\notin \pi\Q$. 
\item Let $(\rho_j)_{j=1}^d\subset(-\infty,0)$ be $\Q$-linearly independent, let $A=\textup{diag}(\exp(\rho_1),...,\exp(\rho_d))$ and let $b\in \C^d$. Then $C_{Az+b}$ is cyclic on $\F(\C^d)$.
\end{itemize}
\end{example}

The proof of Theorem \ref{Theorem 1} will be rather long. We will start in Section~\ref{section 2} by studying two simple cases  which are significant enough to point out the main ideas behind the proof. 
In the same section, we will also expose several lemmas.
In Section~\ref{section 3} we prove the sufficient condition of Theorem~\ref{Theorem 1}
whereas the necessary condition of Theorem~\ref{Theorem 1} will be presented in Section~\ref{section 4}. 
In Section~\ref{section 5} we characterize the set of cyclic vectors for compact composition operators on $\mathcal F(\C^d)$.
The remainder of the paper is devoted to solve several problems
on composition operators on $\mathcal F(\C^d)$ where the techniques introduced
to prove Theorem~\ref{Theorem 1} are useful. 
In Section~\ref{section 6} we show that bounded composition operators on $\F(\C^d)$ are never supercyclic with respect the pointwise convergence topology (and thus neither weakly-supercyclic) nor convex-cyclic. 
In Section~\ref{section 7} we compute the approximation numbers of compact composition operators on $\F(\C^d)$.\\

\noindent\textbf{Notation.} For $z\in \C\setminus\{0\}$, we denote by $\arg(z)\in [0,2\pi)$ the number such that $z= |z|\exp(i\arg(z))$. 
For a matrix $A\in\C^{d\times d}$, its norm is defined by $\|A\|:= \sup \{|Av|:~|v|=1\}$, its transpose is denoted $A^T$ and its Hermitian transpose is denoted by $A^*$.
For $(x_{j})_{j=1}^d\subset\C$, we denote by $\textup{diag}(x_1,...,x_d)$ the diagonal matrix with entries $(x_{j})_{j=1}^d$.
For any $\alpha\in \Z^d$, we denote the length of $\alpha$ by $|\alpha|:= \sum_{i=1}^d |\alpha_i|$ and for any vector $\lambda\in \C^d$ we write $\lambda^\alpha:=\prod_{i=1}^d {\lambda_i}^{\alpha_i}$. 
We also consider the partial order $\leq$ on $\N^d$ defined as follows: $\alpha \leq \beta $ if and only if $\alpha_j\leq \beta_j$ for all $j=1,...,d$. 
By $\D, \overline{\D}$ and $\T$ we respectively denote the open complex unit disc, its closure and its boundary. Finally, when dealing with the adjoint of composition operators,
we will require to introduce weighted composition operators: for $\varphi:\mathbb C^d\to\mathbb C^d$ and $\psi:\mathbb C^d\to \mathbb C$,  the weighted composition operator with symbols $(\psi,\varphi)$ is defined by
\[f\in \mathcal{H}(\C^d)\mapsto W_{\psi,\varphi}(f):= \psi \cdot f\circ \varphi.\] 
Further information about weighted composition operators defined on $\F(\C)$ can be found in \cite{CG, M,M2,M3} and references therein.
\section{Preliminaries}\label{section 2}

\subsection{Two particular cases}

This subsection is purely expository. It aims to point out the main difference between Jordan
blocks of size $2$ and of size $3$. Denote for $d\geq 1$ and $N\geq 0$ by
$\mathcal P_{\textrm{hom}}(N,d)$ the set of all homogeneous polynomials of degree $N$
in $d$ variables, namely
$$\mathcal P_{\textrm{hom}}(N,d)=\textrm{span}(z_1^{\alpha_1}\cdots z_d^{\alpha_d}:\alpha_1+\cdots+\alpha_d=N).$$
Let us recall that $\dim(\mathcal P_{\textrm{hom}}(N,d))=\binom{N+d-1}{d-1}$. Let $P_{N,d}$
be the the orthogonal projection on $\mathcal P_{\textrm{hom}}(N,d)$ in $\mathcal F(\C^d)$. Let us also denote
$$A_1=\begin{pmatrix}1/2&a\\0&1/2\end{pmatrix}\quad A_2=\begin{pmatrix}1/2&a&0\\0&1/2&a\\ 0&0&1/2\end{pmatrix}$$
where $a\in\C\backslash\{0\}$ is so that $C_{A_1z}$ and $C_{A_2z}$ are bounded operators on
$\mathcal F(\C^2)$ and $\mathcal F(\C^3)$ respectively. We are going to show that
\begin{enumerate}[(a)]
\item for any $N$ large enough, for any $f\in\mathcal F(\C^3)$, $\textrm{span}(P_{N,3}(C_{A_2z}^j f):\ j\geq 0)$ is not dense in $\mathcal P_{\textrm{hom}}(N,3)$, which prevents $f$
to be a cyclic vector for $C_{A_2z}$;
\item for any $N\geq 0$, there exists $f\in\mathcal F(\C^2)$ such that $\textrm{span}(C_{A_1z}^j f:\ j\geq 0)$ is equal to $\mathcal P_{\textrm{hom}}(N,2)$.
\end{enumerate}
Let us start with (a) and write $f=\sum_{\alpha}c_\alpha z^\alpha$. Computing $A_2^j$, we easily get that, for all $j\geq 0$, 
$$C_{A_2z}^j f=\sum_{\alpha\in\N^3}\frac{c_\alpha}{2^{j|\alpha|}}(z_1+2jaz_2+2j(j-1)a^2z_3)^{\alpha_1}(z_2+2jaz_3)^{\alpha_2}z_3^{\alpha_3}$$
so that, expanding the product,
$$P_{N,3}(C_{A_2z}^j f)=\frac1{2^{jN}}\sum_{k=0}^{2N}j^k L_k$$
where $L_0,\dots,L_{2N}$ are fixed polynomials in $\mathcal P_{\textrm{hom}}(N,3)$. Therefore
$$\dim(\textrm{span}(P_{N,3}(C_{A_2z}^j f):\ j\geq 0))\leq 2N+1<\dim(\mathcal P_{\textrm{hom}}(N,3))$$
provided $N$ is large enough.

Regarding (b), let $f=\sum_{|\alpha|=N}z^\alpha=\sum_{k=0}^N z_1^k z_2^{N-k}$. Then, for all $j\geq 0$, 
\begin{align*}
C_{A_1z}^j (f)&=\sum_{k=0}^N \frac{1}{2^{Nj}}(z_1+2jaz_2)^k z_2^{N-k}\\
&=\frac1{2^{Nj}}\sum_{k=0}^N j^k L_k
\end{align*}
where $(L_0,\dots,L_N)$ is a basis of $\mathcal P_{\textrm{hom}}(N,2)$. Now, 
$$\frac{2^{Nj}C_{A_1z}^j(f)}{j^N}\to L_N\in \textrm{span}(C_{A_1z}^j f:\ j\geq 0).$$
Hence, 
$$\frac{2^{Nj}C_{A_1z}^j(f)-j^NL_N}{j^{N-1}}\to L_{N-1}\in \textrm{span}(C_{A_1z}^j f:\ j\geq 0)$$
and iterating we get that $\textrm{span}(C_{A_1z}^j f:\ j\geq 0)=\mathcal P_{\textrm{hom}}(N,2)$.

\medskip

The proof of Theorem \ref{Theorem 1} will rely on the two ideas exposed above. 
We will also need a supplementary argument, based on Kronecker's theorem, to handle different eigenvalues. Working with matrices which are maybe not unitarily equivalent to their Jordan form and with affine maps instead of linear ones will cause some extra troubles
which require the introduction of the tools which are described in the remaining part of this section.

\subsection{Useful lemmas}

In this section we collect some facts which will help us in the forthcoming proof of Theorem~\ref{Theorem 1}.

\begin{proposition}\label{product of functions}
	Let $f,g\in \mathcal{F}(\C^d)$ be two functions. Assume that there are two disjoint sets $I_f, I_g\subset \{1,\cdots,d\}$ such that $f(z)= f((z_i)_{i\in I_f})$ and $g(z)= g((z_i)_{i\in I_g})$. Then the function defined by $z\mapsto h(z):=f(z)g(z)$ belongs to $ \mathcal{F}(\C^d)$.
\end{proposition}
\begin{proof}
	It follows from the definition of the norm on $\mathcal{F}(\C^d)$. Indeed $\|h\|=\|f\|\cdot \|g\|$.
\end{proof}
\begin{proposition}\label{conjugated operator}
	Let $\varphi(z):=Az+b$ be such that $C_\varphi$ induces a bounded composition operator on $\mathcal{F}(\C^d)$.
	Then, there are $S\in \C^{d\times d}$ and $v\in\C^d$ such that $C_\varphi$ is similar to $C_{Sz+v}$ and 
	\[
	S= \begin{pmatrix} T& 0  \\ 
	0 & U \\  
	\end{pmatrix},
	\]
	where $T\in\C^{p\times p}$ is an upper triangular matrix such that its diagonal contains all the eigenvalues of $A$ of modulus lower than $1$ and $U\in\C^{(d-p)\times (d-p)}$ is a diagonal matrix containing all the eigenvalues of $A$ of modulus $1$. Moreover $v\in \C^p\times\{0\}^{d-p}$.
\end{proposition}
\begin{proof}
	Let $P,S\in \C^{d\times d}$ be an orthogonal matrix and an upper triangular matrix obtained by the Schur decomposition of $A$, that is, $A= PSP^*$.
	Further, assume that in the first $p$ entries of the diagonal of $S$ we find all the eigenvalues of $A$ of modulus strictly lower than $1$. 
	
	Since ${\|P\|=\|P^*\|=1}$ and $P^{-1}=P^*$, we have that $C_{Pz},C_{P^*z}$ are invertible elements of $\mathcal{L}(\mathcal{F}(\C^d))$, with $C_{Pz}C_{P^*z}=Id$, and that
	
	\[ C_{Sz + P^*b}= C_{Pz}C_\varphi C_{P^*z}\in \mathcal{L}(\mathcal{F}(\C^d)).\] 
	
	Now, noticing that $S$ is an upper triangular matrix, $\|S\|\leq 1$ and the last $d-p$ entries of its diagonal have modulus equal to one, we get
	\[
	S= \begin{pmatrix} T& 0  \\ 
	0 & U \\  
	\end{pmatrix}
	\]
	where $T\in\C^{p\times p}$ is an upper triangular matrix such that its diagonal contains all the eigenvalues of $A$ of modulus lower than $1$ and $U\in\C^{(d-p)\times (d-p)}$ is a unitary diagonal matrix.
	Finally, since $C_{Sz + P^*b}$ is bounded, $v:= P^*b \in \C^p\times\{0\}^{d-p}$.
	
\end{proof}

\begin{lemma}\label{lem:range}
Let $\varphi(z):=Az+b$ be such that $C_\varphi$ induces a bounded composition
operator on $\mathcal F(\C^d)$. Then $b\in\textrm{Ran}(I-A)$. 
\end{lemma}
\begin{proof}
By \cite[Lemma 5.2]{JPZ}, $b\in \ker(I-A^*)^\perp=\textrm{Ran}(I-A)$.
\end{proof}

For the sake of completeness, we state the following two results which are taken from \cite{BS} and \cite{CMS} respectively.

\begin{proposition}\cite[Proposition 2.7]{BS}\label{noncyclicity adjoint}
	Let $H$ be a Hilbert space and let $T\in \mathcal{L}(H)$. If there is $\lambda\in\C$ such that $\textup{dim}(\ker(T^*-\lambda I))\geq 2$, then $T$ is not cyclic.
\end{proposition}
\begin{proposition}\cite[Lemma 2]{CMS}\label{adjoint operator}
	Let $\varphi(z):=Az+b$ be such that $C_\varphi$ induces a bounded composition operator on $\mathcal{F}(\C^d)$. 
	Then $C_\varphi^*=W_{k_b,\widehat{\varphi}}$, that is, the weighted composition operator with symbols $k_b$ and $\widehat{\varphi}(z):= A^*z$.
\end{proposition}

Albeit simple, the following proposition will help us in the forthcoming computations. 
We recall that the symbol $\varphi$ of a bounded composition operator $C_\varphi$ on $\F(\C^d)$ always has a fixed point, \cite[Lemma 5.2]{JPZ}.

\begin{proposition}\label{polynomials of deg 1}
Let $\varphi(z):=Az+b$ be such that $C_\varphi$ induces a bounded composition operator on $\mathcal{F}(\C^d)$. 
Let $\lambda=(\lambda_{j})_{j=1}^d$ be the eigenvalues of $A$ repeated by algebraic multiplicity .
Let $(v_j)_{j=1}^d\subset (\C^d)^*$ be a basis of generalized eigenvectors
of $A^T$ associated to $\lambda$ such that, for all $j=1,\dots,d$, 
either $A^T v_j=\lambda_j v_j$ or $A^T v_j=\lambda v_j+v_{j-1}$. Let $\xi\in\C^d$ 
be a fixed point of $\varphi$. Then there is $L=(L_j)_{j=1}^d$ a basis of 
$\textrm{span}(z_k-\xi_k:\ k=1,\dots,d)$ such that 
\begin{align*}
A^T v_j=\lambda_jv_j ~(\text{or }= \lambda_jv_j+v_{j-1}) \Rightarrow C_\varphi L_j = \lambda_j L_j~ (\text{resp. }=\lambda_j L_j+L_{j-1}),
\end{align*} 
for all $j=1,\dots,d$.
\end{proposition}

\begin{proof}
Observe that each $v_j$, as a linear form on $\C^d$, can also be seen as an element
of $\mathcal F(\C^d)$ and that $C_{Az}(v_j)=A^T v_j$. Now, noticing that
$\varphi(z)=A(z-\xi)+\xi$, we get that the polynomial $L_j:=v_j(\cdot-\xi_j)$ satisfies
$C_\varphi L_j\in\{\lambda_j L_j,\lambda_j L_j+L_{j-1}\}$ and that 
$(L_j)_{j=1}^d$ is a basis of $\textrm{span}(z_k-\xi_k:\ k=1,\dots,d)$.
%
\end{proof}
\begin{remark}
Observe that $\{L^\alpha:\ \alpha\in\mathbb N^d\}$ is a basis of the space of all polynomials in $d$ variables.
\end{remark}

We will also need the following combinatorial lemma (the partial order of $\N^p$ that we consider has been defined at the end of the introduction).

\begin{lemma}\label{lem:combinatorial}
Let $p\geq 1$ and $E\subset\mathbb N^p$. There exists a finite partition $\{D_i: i\in I\}$ of $E$ such that, for all $i\in I$, there exists $\alpha(i)\in D_i$ satisfying $\alpha\geq \alpha(i)$ for all $\alpha\in D_i$. 
\end{lemma}
\begin{proof}
We shall proceed by induction on $p$, the case $p=1$ being trivial. Let $p\geq 2$ and assume that the result has been proven up to $p-1$. Consider any $\beta\in E$ and split $E$ into the finite partition $E_0,\dots,E_p$ with
\begin{align*}
E_0&=\{\alpha\in E:\ \alpha\geq \beta\}\\
E_j&=\{\alpha\in E:\ \alpha_j<\beta_j\}\backslash(E_1\cup\dots\cup E_{j-1}),\ j=1,\dots,p.
\end{align*}
For each $j=1,\dots,p$, we can decompose $E_j$ into the finite partition $E_{j,0},\cdots,E_{j,\beta_{j}-1}$ where 
$$E_{j,k}=\{\alpha\in E_j:\ \alpha_j=k\}.$$
Since one coordinate of each element of $E_{j,k}$ is fixed, one can apply the induction
hypothesis to $E_{j,k}$ to find a finite partition $\{D_{j,k,i}:\ i\in I_{j,k}\}$ of $E_{j,k}$
such that, for all $(j,k,i)$, there exists $\alpha(j,k,i)\in D_{j,k,i}$ satisfying 
\begin{equation}\label{eq:combinatorial}
\forall \alpha\in D_{j,k,i},\ \alpha_l\geq\alpha(j,k,i)_l\textrm{ for }l\in\{1,\dots,p\}\backslash \{j\}.
\end{equation}
Now since for $\alpha\in D_{j,k}$, $\alpha_j=\alpha(j,k,i)_j=k$, \eqref{eq:combinatorial} is true for all $l=1,\dots,p$, namely $\alpha\geq\alpha(j,k,i)$ for all $\alpha\in D_{j,k,i}$. Therefore, $E_0\cup\{D_{j,k,i}:\ j=1,\dots,p,\ k=0,\dots,\beta_j-1,\ i\in I(j,k)\}$ is the partition we are looking for. 
\end{proof}

We will finally require the invertibility of a Vandermonde-like matrix.

\begin{lemma}\label{invertible matrix}
	Let $N\geq 1$. Let $\{\alpha(n):~n=1,...,N\}\subset \Z^d$, where $\alpha(n)=\alpha(m)$ only if $n=m$. 
	Then, there is $\{w(n):~n=1,...,N\}\subset \T^d$ such that the matrix $(w(i)^{\alpha(j)} )_{i,j=1,\dots,N}\in \C^{N\times N}$ is invertible.
\end{lemma}
\begin{proof}
	Let us proceed by induction on $N$. If $N=1$, the result is clear
	and we assume that Lemma~\ref{invertible matrix} holds true for some $N\geq 1$. 
	Let us choose $\{w(i):~i=1,...,N\}\subset \T^d$ such that the matrix $M:=(w(i)^{\alpha(j)} )_{i,j=1,\dots,N}\in \C^{N\times N}$ is invertible. 
	Therefore, $\textup{det}(M)\neq 0$.
	Now, let us consider the function 
	\[z\in \T^d \mapsto f(z):=\det \big({(w(i)^{\alpha(j)})}_{i,j=1,\dots,N+1}\big),~ \]
	where $w(N+1)=z$. Developing the determinant that defines $f(z)$ using the last row,  thanks to the induction hypothesis and the fact that $\alpha(N+1)\neq \alpha(n)$ for $n=1,\dots,N$, we get that $f$ is a trigonometric polynomial with at least one non-zero coefficient. 
	Therefore, there is $z\in \T^d$ such that $f(z)\neq 0$.
\end{proof}

\section{Cyclic composition operators}\label{section 3}
In order to provide the proof of Theorem~\ref{Theorem 1}, we need the following auxiliary results.
\begin{proposition}\label{computation for induction}
	Let $C_\varphi$ be a bounded composition operator on $\F(\C^d)$.
	Let $p\geq 1$ and $\lambda\in \D^{p}$, where $\lambda_{p-1}=\lambda_p$. 
	Assume that there is $L:=(L_i)_{i=1}^p\subset\F(\C^d)$  a finite sequence of polynomials such that $C_\varphi L_i=\lambda_iL_i$ for all $i=1,...,p-1$ and $C_\varphi L_p= \lambda_{p-1}L_p+L_{p-1}$. Then, there is $J\in \N$ such that for any $j\geq J$, for any $n\in\N$ and for any $D\subset \{\alpha\in\N^p:~|\alpha|=n\}$, we have 
	
	\[C_\varphi^j \left(\sum_{\substack{\alpha\in D}} L^\alpha\right)=\sum_{\substack{\alpha\in \N^p\\ |\alpha|=n}}  c(\alpha,D,j) L^\alpha, ~\text{for all }i=1,...,p\]
	where $|c(\alpha,D,j)|\leq 1$.	
	
\end{proposition}

\begin{proof}
Let $j\in \N$. We compute

\begin{align*}
C_\varphi^j \left(\sum_{\substack{\alpha\in D}} L^\alpha\right)&= \sum_{\substack{\alpha\in D}} \prod_{i=1}^p \big(C_\varphi^j(L_i)\big)^{\alpha_i}\\
&= \sum_{\substack{\alpha\in D}} \left(\prod_{i=1}^{p-1} \lambda_i^{\alpha_i j}L_i^{\alpha_i}\right)\big(\lambda_{p-1}^jL_p+j\lambda_{p-1}^{j-1}L_{p-1}\big)^{\alpha_p}\\
&= \sum_{\substack{\alpha\in D}} \lambda^{j\alpha} \left(\prod_{i=1}^{p-1} L_i^{\alpha_i}\right)
\sum_{\beta=0}^{\alpha_p}\begin{pmatrix}
\alpha_p \\ \beta
\end{pmatrix}L_p^\beta \left(\frac{j}{\lambda_{p-1}}\right)^{\alpha_p-\beta}L_{p-1}^{\alpha_p-\beta}\\
&=\sum_{\substack{\alpha\in \N^p\\ |\alpha|=n}} L^\alpha \left(\prod_{i=1}^{p-2}\lambda_i^{j\alpha_i}\right)\lambda_{p-1}^{j(\alpha_{p-1}+\alpha_p)}\sum_{\substack{\gamma\in \N^2\\ |\gamma|=\alpha_{p-1}+\alpha_p\\
(\alpha_1,...,\alpha_{p-2},\gamma)\in D		\\
\gamma_2\geq\alpha_p }}
\begin{pmatrix}
\gamma_2 \\\alpha_{p} 
\end{pmatrix} \left(\frac{j}{\lambda_{p-1}}\right)^{\gamma_2-\alpha_p} 
\end{align*}

Now, let us fix $\alpha\in \N^p$, with $|\alpha|=n$ and set $N:=\alpha_{p}+\alpha_{p-1}$. Observe that

\begin{align*}
\Bigg|\lambda_{p-1}^{jN}\sum_{\substack{\gamma\in \N^2,\ |\gamma|=N\\
(\alpha_1,...,\alpha_{p-2},\gamma)\in D		\\
		\gamma_2\geq\alpha_p }}
\begin{pmatrix}
\gamma_2 \\\alpha_{p} 
\end{pmatrix} \left(\frac{j}{\lambda_{p-1}}\right)^{\gamma_2-\alpha_p}  \Bigg |&\leq |\lambda_{p-1}|^{jN}\sum_{\substack{\gamma\in \N^2\\ |\gamma|=N\\
		\gamma_2\geq\alpha_p }}
\begin{pmatrix}
|\gamma| \\ \alpha_{p}+\gamma_1
\end{pmatrix} \left(\frac{j}{|\lambda_{p-1}|}\right)^{\gamma_2-\alpha_p}\\
&\leq |\lambda_{p-1}|^{jN} \left(1+\frac{j}{|\lambda_{p-1}|}\right)^{N}\\
&=(|\lambda_{p-1}|^j+ j|\lambda_{p-1}|^{j-1})^{N}.
\end{align*}
Since $\lambda_{p-1}\in\D$, there is $J\in \N$ such that  $|\lambda_{p-1}|^j+ j|\lambda_{p-1}|^{j-1}\leq 1$ for all $j\geq J$. Notice that $J$ does not depend on $\alpha$. 
\end{proof}

\begin{lemma}\label{obtaining L alpha}
	Let $0\leq p\leq d$.
	Let $(\lambda,\mu):=(\lambda_j)_{j=1}^p\times(\mu_j)_{j={p+1}}^{d}\in (\D\setminus \{0\})^p\times\T^{d-p}$.
	Let $f:\T^{d-p}\to \F(\C^d)$ be a function. 
	Let $R\in (0,1)$ and $\mcn:=\{\alpha\in\N^p:~|\lambda^\alpha|=R\}$. 
	Let $(x_\alpha)_{\alpha\in \mcn}\subset \F(\C^d)$ be a sequence of linearly independent functions such that, for each $\alpha\in \mcn$ the function $w\in \T^{d-p}\mapsto f(w)x_\alpha\in \F(\C^d)$ is well defined and continuous. 
	Assume that 
	\begin{align*}\tag{H}\label{hypothesis}
	\text{there is no }(\alpha,\beta)\in \Z^p\times\Z^{d-p}\setminus \{\{0\}^d\}~\text{such that }\lambda^\alpha\mu^\beta=1.
	\end{align*}
	Then, for any fixed $\gamma \in \mcn$, the closure of the linear space spanned by the accumulation points in $\F(\C^d)$ of the sequence
	\[  \left(f(\mu_{p+1}^n,...,\mu_{d}^n)\sum_{\alpha\in \mcn} \left(\dfrac{\lambda^{\alpha}}{\lambda^{\gamma}} \right)^n x_\alpha \right)_n \]
	contains the set $\{f(w)x_\alpha:~\alpha\in \mcn,~w\in \T^{d-p}\}$.
	
\end{lemma}

Observe that hypothesis~\eqref{hypothesis} is equivalent to: $\left\{\pi,\arg\left({\lambda^{\alpha}}\mu^{\beta}\right)\right\}$ is $\Q$-linearly independent for all $(\alpha,\beta)\in (\Z^p\times \Z^{d-p})\setminus\{\{0\}\}^d$ such that $|\lambda^\alpha|=1$.

\begin{proof}
	If $\mcn =\emptyset$, there is nothing to prove. 
	So, we assume that $\mcn\neq \emptyset$.
	First, let us write $\lambda_j= e^{\rho_j}e^{i\theta_j}$ for all $j=1,...,p$ and $\mu_j= e^{i\theta_j}$ for all $j=p+1,...,d$, where $(\rho_j)_j,(\theta_j)_j\subset \R$.
	Observe that, for any $\alpha\in \Z^p$, $|\lambda^\alpha|=1$ if and only if $\sum_{j=1}^p \alpha_j\rho_j=0$. 
	Let $E:=\{\alpha\in\Q^d:~\sum_{j=1}^p\alpha_j\rho_j=0\}$. 
	We extract from $\{\rho_1,...,\rho_p\}$ a $\Q$-linearly independent family of maximal cardinality, namely $\{\rho_1,...,\rho_q\}$ and we set $(a_{j,k})_{j,k}\subset \Q$ such that
	
	\[ \rho_k= -\sum_{j=1}^q a_{j,k}\rho_j,~ \text{for all }k=q+1,...,p.\]
	Then, it follows that
	\begin{align}\label{equiv: lemma} 
	 \alpha\in E~\iff ~ \alpha_j=\sum_{k=q+1}^p a_{j,k}\alpha_k,~\text{for all }j=1,...,q.
	\end{align}
	
	\textbf{Claim.} The set $\{\pi\}\cup\{\theta_k+\sum_{j=1}^qa_{j,k}\theta_j:k=q+1,...,p\}\cup\{\theta_k:~k=p+1,...,d\}$ is $\Q$-linearly independent.
	
	 Indeed, otherwise there are $m,(r_k)_k\subset \Z$ such that
	
	\begin{align*}
	0&=m\pi+\sum_{k=q+1}^p r_k\big(\theta_k+\sum_{j=1}^q a_{j,k}\theta_j\big)+\sum_{k=p+1}^d r_k\theta_k
	\end{align*}
	\begin{align}\label{eq: lemma 1}
	&= m\pi + \sum_{j=1}^q \left( \sum_{k=q+1}^p r_ka_{j,k} \right)\theta_j +\sum_{k=q+1}^d r_k\theta_k.
	\end{align} 
	
	Let us define $\alpha\in \Q^d$ by 
	\begin{align*}
	\alpha_j:=\begin{cases}
	 \sum_{k=q+1}^p r_ka_{j,k} & ~\text{if } j=1,...,q.\\
	 r_j & ~\text{if } j=q+1,...,d .
	\end{cases}
	\end{align*}
	Thus, thanks to~\eqref{equiv: lemma}, $\alpha\in E $. However, for some $K\in \N$, $K\alpha\in \Z^d\cap E$. Then \eqref{eq: lemma 1} contradicts assumption~\eqref{hypothesis} and the claim is proved. 
	
	Let us fix $\gamma\in \mcn$.
	Observe that $(\alpha-\gamma)\times\{0\}^{d-p}\in E$ for any $ \alpha\in \mcn$, i.e. 
\begin{equation}\label{eq:alphagamma}
\alpha_j-\gamma_j=\sum_{k=q+1}^p a_{j,k}(\alpha_k-\gamma_k),~\text{for all } j=1,...,q. \end{equation}
 	Now, notice that
 	\begin{align*}
 	g_n:&=f(\mu_{p+1}^n,...,\mu_{d}^n)\sum_{\alpha\in \mcn} \left(\dfrac{\lambda^{\alpha}}{\lambda^{\gamma}} \right)^n x_\alpha\\
 	&=f(e^{in\theta_{p+1}},...,e^{in\theta_{d}})\sum_{\alpha\in \mcn}  x_\alpha \prod_{k=q+1}^p e^{in(\theta_{k} +\sum_{j=1}^q a_{j,k}\theta_j)(\alpha_{k}-\gamma_{k})   }.
 	\end{align*}
 	Therefore, thanks to the above claim and Kronecker's Theorem, we conclude that, for any $w\in \T^{d}$, there is a sequence of integers $(n(l))_l$ such that
 	
 	\[g_{n(l)}\xrightarrow[l\to\infty]{ }  f(w_{p+1},...,w_d)\sum_{\alpha\in \mcn}  x_\alpha \prod_{k=q+1}^p w_k^{\alpha_k-\gamma_k}.\]
 	
 	Finally, Lemma~\ref{obtaining L alpha} follows directly from Lemma~\ref{invertible matrix} and the fact that the function
 	$\alpha\in \mathcal{N}\mapsto (\alpha_k-\gamma_k)_{k=q+1,\dots,p}$ is one-to-one by \eqref{eq:alphagamma}.
\end{proof}

Now we are ready to prove the first half of Theorem~\ref{Theorem 1}.

\begin{proof}[Proof of Theorem~\ref{Theorem 1}: Sufficient condition]
Let $\varphi(z):=Az+b$ be an affine map such that $C_\varphi$ induces a bounded composition operator on $\F(\C^d)$. 
Let us assume that the canonical Jordan form of the invertible matrix $A$ admits exactly one Jordan block of size $2$ and $d-2$ Jordan blocks of size $1$.
Also, we assume that the eigenvalues of $A$ satisfy the hypothesis of the statement of Theorem~\ref{Theorem 1}.
If $A$ is diagonalizable, the proof is completely similar (in fact, simpler). The details of this case are left to the reader.

By Proposition~\ref{conjugated operator}, we can (and shall) assume that $A=\begin{pmatrix} T&0\\ 0&U \end{pmatrix}$, where $T\in \C^{p\times p}$ is an upper triangular matrix and $U\in \C^{(d-p)\times (d-p)}$ is a unitary diagonal matrix, and $b\in \C^p\times \{0\}^{d-p}$. 
Let us call $\lambda\in \C^p$ the diagonal of $T$, i.e. $\lambda$ contains all the eigenvalues of $A$ of modulus lower than $1$ and we further assume that $\lambda_{p-1}=\lambda_p$. \\

Thanks to Proposition~\ref{polynomials of deg 1}, there is $L=(L_i)_{i=1}^p\subset\F(\C^d)$ a finite sequence of linearly independent polynomials of degree $1$ such that 
$C_\varphi L_i(z) = \lambda_i L_i(z)$ for all $i=1,...,p-1$, and $C_\varphi L_p(z)= \lambda_{p-1}L_p + L_{p-1}$.
Observe that, for each $i=1,\dots,p$, the polynomial $L_i$ depends only on $\{z_1,...,z_p\}$. 
Therefore, $\{L^\alpha:\ \alpha\in\N^p\}$ is a basis of the vector space of polynomials on $(z_1,\dots,z_p)$.\\

In order to continue, we define $\rho_0=1$ and for each $k\in \N$, $k\geq 1$:

\[\rho_k := 2^{-k} \Bigg(\sum_{\substack{ \alpha\in \N^p \\ |\alpha|=k} } \| L^\alpha\| \Bigg)^{-1}\wedge \rho_{k-1}.\]

Let us set $w:=\{0\}^{p}\times \{1\}^{d-p}\in \C^d$. Observe that, since $k_w(z)=\exp (\frac{\langle z,w\rangle}{2})$, $k_w$ depends only on $(z_i)_{i=p+1}^d$.
Let us consider the function $h$ defined by

\[ z\in \C^d\mapsto h(z):= k_w(z) \left( \sum_{\substack{ \alpha \in \N^p }} d_{\alpha}L^{\alpha}(z)  \right),\]
where $d_{\alpha}= \rho_{| \alpha| }>0 $ for all $\alpha\in \N^p$. Observe that, thanks to Proposition~\ref{product of functions} and the definition of $( \rho_k)_k$, the function $h$ belongs to $\mathcal{F}(\C^d)$, with $\|h\|\leq 2\|k_w\|$.\\

 We claim that $h$ is a cyclic vector for $C_\varphi$. Let us denote $H:= \overline{\textup{span}}(C^j_\varphi h:~j\in\N)$. In what follows, we proceed by induction to prove that, for every $\alpha \in \N^p$ and every $\widehat{w}\in \{0\}^p\times \T^{d-p}$, $k_{\widehat{w}}L^{\alpha}\in H$. 
A key point will be to understand how the multiindices $\alpha$ are ordered.
 
Let us consider a decreasing enumeration $(R(n))_n$ of the set $\{|\lambda^\alpha|:~\alpha\in\N^p\}$. 
 Also, for $n\in\N$, we define $\mcn(n):= \{\alpha\in\N^p:~| \lambda^\alpha|=R(n)\}$.
 Observe that $R(0)=1$, $\mcn(0)=\{\{0\}^p\}$, that each $\mcn(n)$ is finite and that $\{\mcn(n):n\in\N\}$ is a partition of $\N^p$.
At step $n$, we will show that $k_{\widehat w}L^\alpha\in H$ for all $\alpha\in \mcn(n)$ and all $\widehat{w}\in \{0\}^p\times \T^{d-p}$. 
 
As in Proposition~\ref{adjoint operator}, we write $\widehat{\varphi}(z)=A^*z$. 
This notation allows us to state the following fact which will be used without special mention.

\smallskip

\textbf{Fact.}  $C_{\varphi}^j k_{w}=k_{\widehat{\varphi}^j(w)}$ for all $j\geq 1$.
 Indeed, using Proposition~\ref{adjoint operator}, for any $f\in \F(\C^d)$ we get

\begin{align*}
\langle C_{\varphi} k_{w},f\rangle&=
\langle (W_{k_b,\widehat{\varphi}})^* k_{w},f\rangle = \langle k_w, W_{k_b,\widehat{\varphi}}(f)\rangle \\
&= \langle k_w, k_b f\circ\widehat{\varphi}\rangle= k_b(w)f(\widehat{\varphi}(w))\\
&= k_b(w)\langle k_{\widehat{\varphi}(w)},f\rangle.
\end{align*}
But $k_b(w)=\exp(\langle w,b\rangle/2)= 1$, proving the fact for $j=1$. Inductively, since $\langle b,\widehat{\varphi}^j(w)\rangle=0$ for all $j$, we obtain that $C_{\varphi}^j k_{w}=k_{\widehat{\varphi}^j(w)}$ for all $j\geq 1$.
Here we use that \[\widehat{\varphi}^j(w)=\{0\}^{p}\times (U^*)^j(\{1\}^{d-p})\in \{0\}^{p}\times\T^{d-p}.\]

\smallskip

\textbf{Initialization step.} We prove that $H$ contains the set $\{k_{\widehat{w}}:~\widehat{w} \in \{0\}^p\times \T^{d-p}\}$. 
Let us consider $\{D_i:~i=1,...,p\}$ a partition of $\N^d\setminus \{\{0\}^{d}\}$ such that, for all $i\in \{1,...,p\}$ and all $\alpha \in D_i$, $\alpha_i\geq 1$.
Denote by $e(i)\subset\N^d$ the multi-index satisfying $|e(i)|=1$ and $e(i)_i=1$ for all $i=1,...,d$.
For $j\in\N$, we compute

\begin{align*}
C^j_\varphi h&= k_{\widehat{\varphi}^j(w)} C^j_\varphi \left( d_{0}+\sum_{i=1}^{p}L_i\sum_{\alpha\in D_i}d_{\alpha}L^{\alpha-e(i)}
\right)\\
&= k_{\widehat{\varphi}^j(w)} d_{0}+k_{\widehat{\varphi}^j(w)}\sum_{i=1}^{p-1}\lambda_i^jL_iC^j_\varphi \left(\sum_{\alpha\in D_i}d_{\alpha} L^{\alpha-e(i)}\right)\\
&+k_{\widehat{\varphi}^j(w)}(\lambda_{p-1}^jL_p+j\lambda_{p-1}^{j-1}L_{p-1})C^j_\varphi\left(\sum_{\alpha\in D_{p}}d_{\alpha}L^{\alpha-e(p)}\right)
\\
\end{align*}
We claim that the second and third summand of the last expression tend to $0$ in $\F(\C^d)$ as $j$ tends to $+\infty$. Indeed, fix $i\in\{1,\dots,p\}$ and let $D_i(n)=\{\alpha\in D_i:\ |\alpha|=n\}$. Then by definition of $d_\alpha$ and Proposition~\ref{computation for induction},
\begin{align*}
C_\varphi^j\left(\sum_{\alpha\in D_i}d_\alpha L^{\alpha-e(i)}\right)&=
\sum_{n=1}^{+\infty}\rho_n C_\varphi^j\left(\sum_{\alpha\in D_i(n)}L^{\alpha-e(i)}\right)\\
&=\sum_{n=1}^{+\infty}\rho_n \sum_{\substack{\alpha\in \N^p\\|\alpha|=n-1}}c(\alpha,D_i(n),j)L^\alpha
\end{align*}
with $|c(\alpha,D_i(n),j)|\leq 1$ for $j$ bigger than some $J$, with $J$ independent 
of $i$ and $n$. Now, 
$$\left\|L_i C_\varphi^j\left(\sum_{\alpha\in D_i(n)}L^{\alpha-e(i)}\right)\right\|\leq \sum_{n=1}^{+\infty}\rho_n \sum_{|\alpha|=n}\|L^\alpha\|\leq 1.$$
Taking into account that \[\|k_{\widehat{\varphi}^j(w)}\|=\exp(|\widehat{\varphi}^j(w))|^2/4)=\exp(|w|^2/4),\] 
and since $\lambda_i\in\D$, Proposition~\ref{product of functions} achieves the proof of the claim.

\smallskip
 
Therefore, the sequence $(C^j_\varphi h)_j$ accumulates at the same points that the sequence $(d_0 k_{\widehat{\varphi}^j(w)})_j$ does. 
Observe that $\hat{\varphi}^j(w)= \{0\}^p\times U^{*j}(\{1\}^{d-p})$ where $U^*:= \textup{diag}(\exp(i\theta_{p+1}),...,\exp(i\theta_d))$.
Moreover, thanks to the hypothesis of the eigenvalues of $A$, the set $\{\pi,\theta_{p+1},...,\theta_{d}\}$ is $\Q$-linearly independent. 
Hence, due to Kronecker's Theorem, for any $\widehat{w}\in \{0\}^p\times\T^{d-p}$, there is a sequence $(j(l))_l\subset \N$ such that $(d_0 k_{\widehat{\varphi}^{j(l)}(w)})_l$ converges to $d_0k_{\widehat{w}}$.
This finishes the proof of the initialization step.

\medskip

\textbf{Inductive step.} Let $n\geq 1$ and assume that $k_{\widehat{w}}L^\alpha\in H$ for all $\alpha\in \bigcup \{\mcn(m):~m\leq n-1\}$ and all $\widehat{w}\in \{0\}^p\times \T^{d-p}$. We prove that $k_{\widehat{w}}L^\alpha\in H$ for all $\alpha\in \mcn(n)$ and all $\widehat{w}\in \{0\}^p\times \T^{d-p}$.\\

Let us fix $\widehat{\alpha}\in \mcn(n)$ such that $\widehat{\alpha}_{p-1}$ is maximum among $\alpha_{p-1}$, for $\alpha\in \mcn(n)$.
Also, let $\{D_i:~i\in I\}$ be a finite partition of $\N^p\setminus \bigcup \{\mcn(m):~m\leq n\}$, given by Lemma~\ref{lem:combinatorial}, satisfying the following condition: for each $i\in I$, there is $\alpha(i)\in D_i$ such that for each $\alpha\in D_i$ we have $\alpha\geq \alpha(i)$. 
Let us define
\begin{align*}
g:= h -k_{w}\sum_{m=0}^{n-1}\sum_{\alpha\in \mcn(m)}d_\alpha L^\alpha =k_{w}\sum_{m=n}^{\infty}\sum_{\alpha\in \mcn(m)}d_\alpha L^\alpha . 
\end{align*} 
and notice that, thanks to the induction hypothesis, $g\in H$. 
In order to simplify the notation, let us set $\Lambda=\lambda^{\widehat{\alpha}}$.
Observe that $|\Lambda|=R(n)$.
Thus, for $j\in \N$ we have that
\begin{align}\label{eq: 1}
\dfrac{C^j_\varphi g }{\Lambda^j j^{\widehat{\alpha}_{p-1}}}&=k_{\widehat \varphi^j(w)}\left( \sum_{\alpha\in \mcn(n)}\dfrac{d_\alpha C^j_{\varphi}(L^\alpha)}{\Lambda^j j^{\widehat{\alpha}_{p-1}}}+ \sum_{i\in I}\dfrac{C^j_{\varphi}(L^{\alpha(i)})}{\Lambda^j j^{\widehat{\alpha}_{p-1}}}C^j_{\varphi}\left( \sum_{\alpha\in D_i}d_\alpha L^{\alpha-\alpha(i)}\right)\right)\in H.
\end{align}

Let us check that the second summand of \eqref{eq: 1} tends to $0$ as $j$ tends to infinity. 
Indeed, let  us fix $i\in I$. Then
\begin{align*}
\dfrac{C^j_\varphi(L^{\alpha(i)})}{\Lambda^j j^{\widehat{\alpha}_{p-1}}}&= \dfrac{\lambda^{j\alpha(i)}}{\Lambda^j j^{\widehat{\alpha}_{p-1}}} \left(L_p+\dfrac{j}{\lambda_{p-1}}L_{p-1}\right)^{\alpha(i)_p}\prod _{m=1}^{p-1}L_m^{\alpha(i)_m} \\
&= \left(\dfrac{\lambda^{\alpha(i)}}{\Lambda}\right)^j\dfrac{1}{j^{\widehat{\alpha}_{p-1}}}\sum_{\beta=0}^{\alpha(i)_p}\begin{pmatrix}
\alpha(i)_p \\ \beta 
\end{pmatrix} \left(\dfrac{j}{\lambda_{p-1}}\right)^{\alpha(i)_p-\beta} L_{p-1}^{\alpha(i)_p-\beta}L_p^\beta \prod _{m=1}^{p-1}L_m^{\alpha(i)_m}\\
&=:\sum_{\beta=0}^{\alpha(i)_p} a(i,j,\beta) L_{p-1}^{\alpha(i)_p-\beta}L_p^\beta \prod _{m=1}^{p-1}L_m^{\alpha(i)_m},
\end{align*}
where $(a(i,j,\beta))_{i,j,\beta}$ are the respective coefficients.
By definition of $R(n)$ and $\mcn(n)$, we have that $|\lambda^{\alpha(i)}|< R(n)= |\Lambda|$. 
Therefore, all the coefficients $a(i,j,\beta)$ of the above expression tend to $0$ as $j$ tends to infinity, whatever the value of $\widehat{\alpha}_{p-1}$.
It is now straightforward to modify the proof of the initialization step to show that

\begin{align*}
k_{\widehat \varphi^j(w)}\sum_{i\in I}\dfrac{C^j_{\varphi}(L^{\alpha(i)})}{\Lambda^j j^{\widehat{\alpha}_{p-1}}}C^j_{\varphi}\left( \sum_{\alpha\in D_i}d_\alpha L^{\alpha-\alpha(i)}\right)
&\xrightarrow[j\to\infty]{ }  0.
\end{align*}
Thus, the sequence $(C^j_\varphi g/ \Lambda^j j^{\widehat{\alpha}_{p-1}})$ accumulates at the same points as the first sum of \eqref{eq: 1}.
Now, observe that
\begin{align*}
\sum_{\alpha\in \mcn(n)} d_\alpha C^j_\varphi (L^\alpha)&=\sum _{\alpha\in \mcn(n)}d_\alpha \lambda^{j\alpha} \left(L_p+\dfrac{j}{\lambda_{p-1}}L_{p-1}\right)^{\alpha_p}\prod_{i=1}^{p-1} L_i^{\alpha_i}\\
&=\sum_{\alpha\in \mcn(n)}\sum_{\beta=0}^{\alpha_p}d_\alpha \begin{pmatrix}
\alpha_p\\ \beta 
\end{pmatrix}\lambda^{j\alpha}\left(\dfrac{j}{\lambda_{p-1}}\right)^{\alpha_p-\beta}L_{p-1}^{\alpha_{p-1}+\alpha_p-\beta} L_p^\beta\prod_{i=1}^{p-2} L_i^{\alpha_i}.
\end{align*}
Rearranging the last expression and recalling that $d_\alpha= \rho_{|\alpha|}$, we get
\begin{align*}
\sum_{\alpha\in \mcn(n)} d_\alpha C^j_\varphi (L^\alpha)&=\sum_{\alpha\in \mcn(n)}d_\alpha \lambda^{j\alpha} L^\alpha \left(\sum_{\beta=\alpha_p}^{\alpha_{p-1}+\alpha_p}\begin{pmatrix}
\beta \\ \alpha_p 
\end{pmatrix}  \left(\dfrac{j}{\lambda_{p-1}} \right)^{\beta-\alpha_p}\right).
\end{align*}
So, for all $\alpha\in \mcn(n)$, the coefficient that multiplies $L^\alpha$ tends to $0$ as the same rate as $R(n)^jj^{\alpha_{p-1}}$. Let us consider now
\[\mcn(n,m):=\{\alpha\in \mcn(n):~\alpha_{p-1}=m\}.\]
It follows that $\{\mcn(n,m):~m=0,...,\widehat{\alpha}_{p-1}\}$ is a partition of $\mcn(m)$. Also, the accumulation points of the sequence $(C^j_{\varphi}g/ \Lambda^jj^{\widehat{\alpha}_{p-1}})$ coincide with the accumulation points of the sequence
\[\left(k_{\widehat \varphi^j(w)} \sum_{\alpha\in \mcn(n,\widehat{\alpha}_{p-1})}
d_\alpha \begin{pmatrix}
\alpha_{p-1}+\alpha_p\\ \alpha_{p} 
\end{pmatrix}\left(\dfrac{\lambda^{\alpha} }{\lambda^{\widehat{\alpha}}}\right)^j L^\alpha\right)_j.\]
Thanks to the hypothesis of the eigenvalues of $A$ and Lemma~\ref{obtaining L alpha}, we get that 
\begin{align*}
\{ k_{\widehat{w}}L^{\alpha}:~\alpha\in \mcn(n,\widehat{\alpha}_{p-1}),~\widehat{w}\in \{0\}^{p}\times\T^{d-p}\}\subset H.
\end{align*}
Inductively, we obtain that for all $m=0,\dots,\widehat{\alpha}_{p-1}$, 
\begin{align*}
\{ k_{\widehat{w}}L^{\alpha}:~\alpha\in \mcn(n,m),~\widehat{w}\in \{0\}^{p}\times\T^{d-p}\}\subset H.
\end{align*}
Indeed, let us assume that the last inclusion holds true for $m=M+1,...,\widehat{\alpha}_{p-1}$. To show that it also holds true for $m=M$, we proceed as above but considering the sequence
\begin{align*}
\dfrac{1}{\Lambda^j j^{M}}\left(C_\varphi^j g - \sum_{m=M+1}^{\widehat{\alpha}_{p-1}}\sum_{\alpha\in \mcn(n,m)}d_\alpha \lambda^{j\alpha} L^\alpha \left(\sum_{\beta=\alpha_p}^{\alpha_{p-1}+\alpha_p}\begin{pmatrix}
\beta \\ \alpha_p 
\end{pmatrix}  \left(\dfrac{j}{\lambda_{p-1}} \right)^{\beta-\alpha_p}\right) \right) \in H,~\forall j\geq 1.
\end{align*} 

\medskip

\textbf{Conclusion.} To conclude the proof, one only need to show that $H=\F(\C^d)$.
Since $\{L^\alpha:~\alpha\in\N^p\}$ is a basis of the vector space of polynomials on $(z_1,..., z_p)$, we have proved that 
\[\{z^\alpha k_{\widehat{w}}:~ \alpha\in \N^{p}\times\{0\}^{d-p},~\widehat{w}\in\{0\}^p\times\T^{d-p}\}\subset H.\]
Let $f\in \F(\C^d)$ be such that $\langle f,g\rangle =0$ for all $g\in H$.
Let us write $f(z):= \sum_{\alpha\in \N^d}a_\alpha z^\alpha$. We know that, for any $\widehat{w}\in \{0\}^p\times\C^{d-p} $, we have  
\[k_{\widehat{w}}(z)= \sum_{\alpha\in \{0\}^p\times\C^{d-p}} c_\alpha z^\alpha,\]
for some sequence $(c_\alpha)_\alpha\subset\C$ depending on $\widehat{w}$. 
Let us fix $\beta\in \N^p\times\{0\}^{d-p}$ and let $P:\N^d\to\N^p\times \{0\}^{d-p}$ be the canonical projection onto the first $p$ coordinates. 
Also, let us consider the function $f_\beta\in \HH(\C^d)$ defined by
\[\sum_{\substack{ \alpha \in \N^d\\
		P(\alpha)=\beta }}a_\alpha z^\alpha =z^\beta\sum_{\substack{ \alpha \in \N^d\\
		P(\alpha)=\beta }}a_\alpha z^{\alpha-\beta} =:  z^\beta f_\beta(z).\]
Observe that $f_\beta$ and $k_{\widehat{w}}$ only depend on $(z_{p+1},...,z_d)$. 
Then, it follows from the orthogonality of the monomials $\{z^\alpha:~\alpha\in\N^p\}$ that
\begin{align*}
0&= \langle f, z^\beta k_{\widehat{w}}\rangle =  \langle z^\beta f_\beta, z^\beta k_{\widehat{w}}\rangle= \|z^\beta\|^2\overline{f_\beta(\widehat{w})}.
\end{align*}
Thus, $f_\beta$ vanishes on $\{0\}^{p}\times \T^{d-p}$. Since $f_\beta$ is an entire function depending only on the last $d-p$ coordinates, we conclude that $f_\beta\equiv 0$. 
Therefore, $a_\alpha=0$ for all $P(\alpha)=\beta$, where $\beta$ is any arbitrary multi-index in $\N^p\times\{0\}^{d-p}$. 
This yields that $f\equiv 0$ and the proof of cyclicity of $C_\varphi$ is complete. 
\end{proof}

\section{Non-cyclic composition operators}\label{section 4}

We split the proof of the necessary condition of Theorem~\ref{Theorem 1} in the following four propositions.
\begin{proposition}\label{non-cyclicity 1}
	Let $A\in \C^{d\times d}$ be a non-invertible matrix and let $b\in\C^d$ be such that $C_{Az+b}$ induces a bounded composition operator on $\F(\C^d)$. Then $C_{Az+b}$ is not cyclic.
\end{proposition}
\begin{proof}
	Since cyclicity is stable under conjugacy, let us assume that $A$ and $b$ have the form given by Proposition~\ref{conjugated operator}. 
We may even assume that the eigenvalue $0$ is placed at the first position of the diagonal of $A$. This implies that the first column of $A$ only has $0$'s.
	Therefore, for any $j\in \N$, with $j\geq 1$, the vector $\varphi^j (z)$ does not depends on $z_1$. 
	Thus, for any $f\in \F(\C^d)$ and $j\geq 1$, the function $C_\varphi^j f$ depends only on $(z_i)_{i=2}^d$. 
	Hence, $f$ cannot be cyclic for $C_\varphi$. 
	Since $f$ is arbitrary, $C_\varphi$ is not a cyclic operator.
	
\end{proof}

\begin{proposition}\label{non-cyclicity 2}
	Let $A\in \C^{d\times d}$ be an invertible matrix and let $b\in\C^d$ be such that $C_{Az+b}$ induces a bounded composition operator on $\mathcal F(\C^d)$. 
	Let $\lambda:=(\lambda_1,...,\lambda_n)\in \overline{\D}^n$ be the eigenvalues of $A$ repeated by geometric multiplicity, $1\leq n\leq d$. 
	If there is $\alpha\in \Z^n\setminus\{0\}^n$ such that $\lambda^\alpha=1$, then $C_{Az+b}$ is not cyclic.
\end{proposition}

\begin{proof}
 By Proposition~\ref{polynomials of deg 1}, let us consider $L_1,\dots,L_n$ be $n$ linearly independent polynomials of degree $1$ such that $C_{\widehat\varphi} L_j=\overline{\lambda_j} L_j$. 
Let $c\in\C^d$ be such that $(I-A)c=\frac b2$ (see Lemma \ref{lem:range}). 
Then for any $\alpha\in\N^n$, the function $z\mapsto L^\alpha e^{\langle z,c\rangle}$ (which belongs to $\mathcal F(\C^d)$ as a product of an exponential function with a polynomial) is an eigenvector
of $C_{\varphi}^*=M_{k_b}C_{\widehat\varphi}$ associated to $\overline{\lambda}^\alpha$. Indeed, 
	\begin{align*}
	C^*_\varphi (L^\alpha e^{\langle z,c\rangle})&= k_b(z)C_{\widehat{\varphi}}(L^\alpha e^{\langle z,c\rangle})=e^{\frac{\langle z,b\rangle }{2}}\overline{\lambda}^\alpha L^\alpha e^{\langle A^*z,c\rangle}=\overline{\lambda}^\alpha L^\alpha e^{\langle z, Ac+\frac{b}{2}\rangle}=\overline{\lambda}^\alpha L^\alpha e^{\langle z,c\rangle}.
	\end{align*}
Suppose now that $\alpha\in\Z^d\backslash\{0\}^d$ satisfy $\lambda^\alpha=1$. If $\alpha\in \N^d$, then the functions $\{L^{n\alpha}e^{\langle z,c\rangle}:\ n\geq 0\}$ are linearly independent eigenvectors of $C_\varphi^*$ associated to the eigenvalue $1$. Thus by Proposition~\ref{noncyclicity adjoint}, $C_\varphi$ is not cyclic. If $\alpha\in\Z^d\backslash \N^d$, let $\alpha^+=(\max(\alpha_j,0))_j$ and $\alpha^-=(-\min(\alpha_j,0))_j$
so that $\alpha^+,\alpha^-\in\N^d$ and $\lambda^{\alpha^+}=\lambda^{\alpha^-}$. 
Now, $L^{\alpha^+}e^{\langle z,c\rangle}$ and $L^{\alpha^-}e^{\langle  z,c\rangle}$ are two linearly independent eigenvectors of $C_\varphi^*$ associated to the same eigenvalue $\overline{\lambda}^{\alpha^+}$. Again, Proposition~\ref{noncyclicity adjoint} provides the conclusion.
\end{proof}

In order to proceed with the remaining cases of non-cyclic composition operators on $\F(\C^d)$, we need the following proposition.

\begin{proposition}\label{Projection and complementability}
Let $\xi\in \C^d$ and $N\in \N$. Let $\PP_N:=\textup{span}\{(z-\xi)^\alpha:~\alpha\in\N^d,~|\alpha|=N\}$ and  $\QQ_N:=\overline{\textup{span}}\{(z-\xi)^\alpha:~\alpha\in\N^d,~|\alpha|\neq N\}$. 
Then, the linear map $P_N$, defined by
\[f\in\F(\C^d)\mapsto P_N(f)(z):= \dfrac{1}{2\pi}\int_0^{2\pi} f(e^{i\theta}(z-\xi)+\xi)e^{-iN\theta}d\theta,\]
is a bounded projection onto $\PP_N$ parallel to $\QQ_N$. In particular, $\F(\C^d)=\PP_N\oplus \QQ_N$.
\end{proposition}

\begin{proof}

Let us start showing that $P_N$ is bounded. 
Let $p\in \N$ and let $f=\sum_{\alpha\in\N^d}c_\alpha z^\alpha$ be such that $c_\alpha=0$ for all $|\alpha|>p$. 
Recall that \[\|f\|^2=\sum_{|\alpha|\leq p}2^{|\alpha|}|c_\alpha|^2\prod_{j=1}^d\alpha_j!.\]

Now we compute
\begin{align*}
P_N(f)(z)&=\sum_{|\alpha|\leq p}\dfrac{c_\alpha}{2\pi}\int_0^{2\pi } e^{-iN\theta} \prod_{j=1}^d(e^{i\theta}(z_j-\xi_j )+\xi_j )^{\alpha_j}d\theta  \\
& =\sum_{|\alpha|\leq p} \dfrac{c_\alpha}{2\pi}\int_0^{2\pi }\sum_{\beta\leq \alpha} e^{-iN\theta} e^{i\theta|\beta|}\prod_{j=1}^d\begin{pmatrix}
\alpha_j \\ \beta_j
\end{pmatrix}(z_j-\xi_j )^{\beta_j}\xi_j ^{\alpha_j-\beta_j}d\theta  \\
&=\sum_{|\alpha|\leq p}\sum_{\substack{|\beta|=N \\ \beta\leq \alpha} }c_\alpha\prod_{j=1}^d\begin{pmatrix}
\alpha_j \\ \beta_j
\end{pmatrix}(z_j-\xi_j )^{\beta_j}\xi_j ^{\alpha_j-\beta_j}\\
&=\sum_{|\beta|=N}\sum_{\alpha\geq \beta}c_\alpha
(z-\xi)^\beta \xi^{\alpha-\beta}
\prod_{j=1}^d\begin{pmatrix}
\alpha_j \\ \beta_j
\end{pmatrix}.
\end{align*}
Fix any $\beta\in\N^d$ such that $|\beta|=N$. 
Since there are finitely many $d$-tuples of size $N$, in order to prove that $P_N$ is bounded we just need to find $C\geq 0$ such that  
\[\sum_{\alpha\geq \beta}|c_\alpha||\xi^{\alpha-\beta}|\prod_{j=1}^d\begin{pmatrix}
\alpha_j \\ \beta_j
\end{pmatrix} \leq C\| f\|.\]
In fact, considering $M=|\xi|$ and the Cauchy-Schwarz inequality, we have that
\begin{align*}
\sum_{\alpha\geq \beta}|c_\alpha||\xi^{\alpha-\beta}|\prod_{j=1}^d\begin{pmatrix}
\alpha_j \\ \beta_j
\end{pmatrix}&\leq \sum_{\alpha \geq \beta}|c_\alpha| \alpha^\beta M^{|\alpha|}\\
& \leq \left( \sum_{\alpha\geq \beta }|c_\alpha|^2 2^{|\alpha|}\prod_{j=1}^d\alpha_j! \right)^{1/2}\left(\sum_{\alpha\geq \beta }\left(\dfrac{M^2}{2}\right)^{|\alpha|} \prod_{j=1}^{d}\frac{\alpha_j^{2\beta_j}}{\alpha_j!} \right)^{1/2}\leq C\|f\|,
\end{align*}
where $C<\infty$.
Thus, $P_N$ is a bounded linear operator on $\F(\C^d)$. \\

Now, by definition of $P_N$, it easily follows that 

\[P_N((z-\xi)^\alpha)=\begin{cases}
(z-\xi)^\alpha&~\text{if } |\alpha|=N\\
0 &~\text{if } |\alpha|\neq N.
\end{cases}\]
Therefore $\PP_N\subset \textup{Ran}(P_N)$. 
In fact, there is equality. 
Indeed, let $f\in \F(\C^d)$ and $\varepsilon>0$.
Since $(z^\alpha)_{\alpha\in\N^d}$ is an orthogonal basis of $\F(\C^d)$ and $\textup{span}\{z^{\alpha}:~|\alpha|\leq q\}$ coincides with $\textup{span}\{(z-\xi)^{\alpha}:~|\alpha|\leq q\}$, for all $q\in\N$, we know that there is $r\geq N$ and $(c_\alpha)_{|\alpha|\leq r}$ such that

\[ \big\| f- \sum_{|\alpha|\leq r} c_\alpha (z-\xi)^{\alpha}\big \|\leq \dfrac{\varepsilon}{\|P_N\|}. \]
Therefore, 
\[\big \| P_N(f)-\sum_{|\alpha|= N} c_\alpha (z-\xi)^{\alpha}\big \|<\varepsilon,\]
which implies that $\textup{Ran}(P_N)\subset \overline{\PP_N}=\PP_N$ since $\PP_N$ is finite dimensional. 

Now, we show that $\QQ_N=\ker (P_N)$. 
We already know that $\QQ_N\subset \ker (P_N)$. 
Conversely, if $P_N(f)=0$, approximating $f$ by a polynomial $\sum_{|\alpha|\leq r}c_\alpha (z-\xi)^\alpha$ as above, we know that
\[ \big\|\sum_{|\alpha|=N} c_\alpha(z-\xi)^\alpha\big\|\leq \varepsilon,\]
which implies that
\[\big\|f-\sum_{\substack{|\alpha|\leq r \\ |\alpha|\neq N}}c_\alpha(z-\xi)^\alpha\big\| \leq 2\varepsilon.\]
Hence, $f\in \QQ_N$. 
\end{proof}

\begin{proposition}\label{2 jordan blocks}
	Let $\varphi(z):=Az+b$ be such that $C_\varphi$ induces a bounded operator on $\F(\C^d)$. If the canonical Jordan form of $A$ admits two Jordan blocks of size $2$, then $C_\varphi$ is not cyclic.
\end{proposition}

\begin{proof}
 Let us assume that the canonical Jordan form of $A$ admits two Jordan blocks of size $2$ associated to the eigenvalues $\lambda_1$ and $\lambda_2$.
 Let $\xi\in \C^d$ be a fixed point of $\varphi$.
 In particular, thanks to Proposition~\ref{polynomials of deg 1}, there are four linearly independent polynomials of degree one $(L_j)_{j=1}^4\subset \F(\C^d)$  such that, for $j\in \{1,2\}$,
 \begin{align*}
	 C_{\varphi} L_{2j-1} &= \lambda_j L_{2j-1}+L_{2j} \\
	 C_{\varphi} L_{2j} &=  \lambda_j L_{2j}.
 \end{align*}

For $N\geq 0$, consider $\PP_N$, $\QQ_N$ and $P_N$ as in Proposition~\ref{Projection and complementability} associated to $\xi$.
Thanks to Proposition~\ref{Projection and complementability}, $\F(\C^d)=\PP_N\oplus \QQ_N$.

Again thanks to Proposition~\ref{polynomials of deg 1}, we fix $(L_j)_{j=5}^d\subset \F(\C^d)$ be linearly independent polynomials of degree $1$ such that $\{L_j:~j=1,...,d\}$ is a basis of $\PP_1$ and, for each $j=5,...,d$, $C_{\varphi}L_j$ belongs to $\textrm{span}(L_k:~k=5,\dots,d)$.
Observe that $\{L^\alpha:~\alpha\in\N^d,~|\alpha|=N\}$ is a basis of $\PP_N$. \\

Let now $n,m\geq 2$. Set $N=n+m$ and define 
\begin{align*}
Y_{n,m}&= \textup{span}\{L_1^kL_2^{n-k}L_3^lL_4^{m-l}:~k=0,...,n,~l=0,...,m\}\\
Z_{n,m}&= \textup{span}\{L^\alpha:~\alpha\in \N^d,~|\alpha|=N,~L^\alpha\notin Y_{n,m}\}.
\end{align*}

\smallskip

\textbf{Fact.} $\PP_N$, $\QQ_N$, $Y_{n,m}$ and $Z_{n,m}$ are $C_\varphi$-invariant subspaces. \\
Indeed, this easily follows from the values of $C_\varphi L_j$, for $j=1,...,d$.

\medskip

Now, let $R_{n,m}:\PP_N\to Y_{n,m}$ be the linear projection associated to $\PP_N= Y_{n,m}\oplus Z_{n,m}$. 
Let us check that the following expression holds true:
\begin{align}\label{eq: conmutative}
R_{n,m}\circ P_N\circ C_\varphi= C_\varphi \circ R_{n,m}\circ P_N.
\end{align} 
Indeed, let $f\in \textup{span}(L^\alpha:~\alpha\in \N^d)$, namely, $f=\sum_{\alpha\in\N^d}c_\alpha L^\alpha$, where there are only finitely many $c_\alpha$ different from $0$. 
Then, thanks to the previous fact we get:
\begin{align*}
R_{n,m}\circ P_N\circ C_\varphi(f)&=R_{n,m} \sum_{\substack{ |\alpha|=N }}c_\alpha C_\varphi(L^\alpha)\\
&= \sum_{\substack{ |\alpha|=N \\ \alpha_1+\alpha_2=n \\
\alpha_3+\alpha_4=m }}c_\alpha \lambda_1^n\lambda_2^m \left(L_1+\dfrac{1}{\lambda_1}L_2\right)^{\alpha_1} L_2^{\alpha_2}\left (L_3+\dfrac{1}{\lambda_2}L_3\right)^{\alpha_3} L_4^{\alpha_4} \\
&=  C_\varphi \circ R_{n,m}\circ P_N (f).
\end{align*}

We are now ready to prove that $C_\varphi$ is not cyclic. 
Pick any $f\in \mathcal{F}(\C^d)$ and write it $f=\sum_{|\alpha|=N}c_\alpha L^\alpha+g$ with $g\in\mathcal Q_N$. Let us call $c_{k,l}=c_{(k,n-k,l,m-l)\times\{0\}^{d-4}}$. 
Thanks to \eqref{eq: conmutative}, for any $j\geq 0$, we have that
\begin{align*}
R_{n,m}\circ P_N\circ C_\varphi^j (f)&= \sum_{k=0}^{n}\sum_{l=0}^m c_{k,l}C_\varphi^j(L_1^kL_2^{n-k}L_3^lL_4^{m-l})\\
&= \sum_{k=0}^{n}\sum_{l=0}^m c_{k,l}\lambda_1^{jn}\lambda_2^{jm}L_2^{n-k}L_4^{m-l}\left(L_1+\frac{j}{\lambda_1}L_2\right)^k\left(L_3+\dfrac{j}{\lambda_2}L_4\right)^l\\
&=(\lambda_1^n\lambda_2^m)^j\sum_{r=0}^{n+m}j^rf_r,
\end{align*}
where $(f_r)_r\subset Y_{n,m}$ are some fixed polynomials that do not depend of $j$.
Therefore, the dimension of $\textup{span}\{R_{n,m}\circ P_N\circ C_\varphi^j(f): ~j\geq 0\}$ is at most $n+m+1$. It cannot be dense in $R_{n,m}\circ P_N(\F(\C^d))=Y_{n,m}$ which has dimension $(n+1)(m+1)$, for instance if $n=m=2$. 
\end{proof}

Now, we proceed with the last case.
\begin{proposition}\label{1 jordan block}
	Let $\varphi(z):=Az+b$ be such that $C_\varphi$ induces a bounded operator on $\F(\C^d)$. If the canonical Jordan form of $A$ admits a Jordan block of size larger than or equal to $3$, then $C_\varphi$ is not cyclic.
\end{proposition}

Since we apply a technique that follows the lines of the proof of Proposition~\ref{2 jordan blocks}, we only present a sketch of the proof of Proposition~\ref{1 jordan block}.
\begin{proof}
Let us assume that the canonical Jordan form of $A$ admits a Jordan block of size $p\geq 3$. 
Let $\xi\in\C^d$ be a fixed point of $\varphi$.
Let $\{L_j:~j=1,...,d\}\subset \F(\C^d)$ be a linearly independent set of polynomials of degree $1$ given by Proposition~\ref{polynomials of deg 1} such that $C_{\varphi}L_j=\lambda_1 L_{j}+L_{j+1}$ for all $j=1,...,p-1$ and $C_{\varphi}L_p=\lambda_1 L_p$.\\

Let $N\in \N$ and consider $\PP_N$, $\QQ_N$ and $P_N$ as in Proposition~\ref{Projection and complementability} associated to $\xi$.
Let us now define
\begin{align*}
Y_N&=\textup{span}\Big\{\prod_{j=1}^pL^{k_{j}}_j:~ \sum_{j=1}^pk_j= N\Big\},\\
Z_N&=\textup{span}\Big\{L^\alpha:~ |\alpha|=N,~ L^\alpha\notin Y_N\Big\}.
\end{align*}
It follows that $\PP_N$, $\QQ_N$, $Y_N$ and $Z_N$ are $C_\varphi$-invariant. 
Let us define $R_N:\PP_N\to Y_N$ be the linear bounded projection associated to $\PP_N=Y_N\oplus Z_N$. 
Moreover, as in \eqref{eq: conmutative}, we have that
\begin{align}\label{eq: conmutative 2}
R_N\circ P_N\circ C_\varphi= C_\varphi \circ R_N\circ P_N.
\end{align} 

Now, let us prove that $C_\varphi$ is not cyclic. 
Indeed, pick any $f\in \F(\C^d)$, with $P_N(f)=\sum_{|\alpha|=N}c_\alpha L^\alpha$ and observe that, for any $j \geq p-1$, we have
\begin{align*}
R_N\circ P_N\circ C_\varphi^j (f)&=
\sum_{\substack{\alpha\in \N^p\times\{0\}^{d-p}\\ |\alpha|=N} } c_\alpha C_\varphi^j(L^\alpha)\\
&= \sum_{\substack{\alpha\in \N^p\times\{0\}^{d-p}\\ |\alpha|=N} } c_\alpha \lambda_1^{jN}\prod_{k=1}^p \left(\sum_{l=k}^p \begin{pmatrix}
j\\
l-k
\end{pmatrix} \dfrac{1}{\lambda_1^k}L_l   \right)^{\alpha_k}\\
& = \lambda_1 ^{jN} \sum_{m=0}^{N(p-1)}j^m f_m,
\end{align*}
where $\{f_m:~m=0,...,N(p-1)\}\subset Y_N$ are some fixed polynomials that do not depend of $j$.
Therefore, the dimension of $\textup{span}\{R_N\circ P_N \circ C^j_\varphi(f):~j\geq 0\}$ is at most $N(p-1)+p-1$. 
It cannot be dense in $R_N\circ P_N(\F(\C^d))= Y_N$ which has dimension $\begin{pmatrix}
N+p-1\\ p-1
\end{pmatrix} $, for instance if $N=3$.

\end{proof}

\begin{proof}[Proof of Theorem~\ref{Theorem 1}: Necessary condition] Let us proceed by a contrapositive argument. Observe that Proposition~\ref{non-cyclicity 1}, Proposition~\ref{non-cyclicity 2}, Proposition~\ref{2 jordan blocks} and Proposition~\ref{1 jordan block} cover all the possible cases of the necessary condition of Theorem~\ref{Theorem 1}. 
Thus, the proof of Theorem~\ref{Theorem 1} is now complete.
\end{proof}

\section{Cyclic vectors of compact composition operators}\label{section 5}

In this section we characterize the set of cyclic vectors for compact cyclic composition operators defined on $\F(\C^d)$. 
In order to state the main result of this section, we need to fix some notations.
Let us consider $\varphi(z):=Az+b$ such that $\|A\|<1$ and let $\xi\in \C^d$ be a fixed point of $\varphi$. 
Also, for any $N\in \N$, the subspace $\PP_N$ and the projection $P_N$ are given by Proposition~\ref{Projection and complementability}. 
Set $L=(L_j)_{j=1}^d\subset\F(\C^d)$ be the polynomials of degree $1$ given by Proposition~\ref{polynomials of deg 1} related to $\varphi$ and $\xi$. 
Recall that the set $\{L^\alpha:~\alpha\in \N^d,~|\alpha|=N\}$ is a basis of $\PP_N$. Thus, for any $f\in \F(\C^d)$, by considering the power series of $f$ centered at $\xi$, there is a unique sequence $(f_\alpha)_{\alpha\in \N^d}\subset \C$ such that $f(z)= \sum_{n=0}^\infty\sum_{|\alpha|=n}f_\alpha L^\alpha(z)$ for all $z\in \C^d$.

\begin{theorem}\label{theorem cyclic vectors}
	Let $\varphi(z):=Az+b$ be such that $C_\varphi$ induces a compact cyclic composition operator on $\F(\C^d)$.
	The following assertions hold true.
	\begin{enumerate}
		\item If $A$ is diagonalizable, then $f\in \F(\C^d)$ is a cyclic vector for $C_\varphi$ if and only if $f_\alpha\neq 0$ for all $\alpha\in \N^d$.
		\item If $A$ is not diagonalizable (and therefore its canonical Jordan form admits a block of size $2$), 
		and if $(L_j)_{j=1}^d$ is ordered so that $(L_j)_{j=1}^{d-1}$ are eigenvectors of $C_\varphi$ and $C_\varphi L_d \in \textup{span}(L_{d-1},L_d)$, then $f\in \F(\C^d)$ is a cyclic vector for $C_\varphi$ if and only if $f_\alpha\neq 0$ for all $\alpha\in \N^d$ with $\alpha_{d-1}=0$.
	\end{enumerate}
	
\end{theorem}

In order to prove Theorem~\ref{theorem cyclic vectors} we need several intermediate results.
\begin{proposition}\label{computing projection}
	For any $N\in \N$ and any $f\in \F(\C^d)$, $P_N(f)= \sum_{|\alpha|=N}f_\alpha L^\alpha$.
\end{proposition}

\begin{proof}
	This easily follows from the following facts:
	\begin{itemize}
		\item $P_N(\mathcal F(\C^d))=\PP_N=\textup{span}\{(z-\xi)^\alpha:~|\alpha|=N\}=\{L^\alpha:~|\alpha|=N\}$,
		\item $(I-P_N)(\mathcal F(\C^d))=\overline{\textup{span}}\{(z-\xi)^\alpha:~|\alpha|\neq N\}$ and 
		\item if $(f_n)_n\subset \F(\C^d)$ converges to $f\in \F(\C^d)$, then for any $\alpha\in\N^d$, the $\alpha$-partial derivative of $(f_n)_n$ converges to the $\alpha$-partial derivative of $f$ for the locally uniform convergence topology.
	\end{itemize}
\end{proof}

From now on, let us further assume that the canonical Jordan form of $A$ admits only one Jordan block whose size is exactly $2$. 
Set $\lambda=(\lambda_j)_{j=1}^{d}\subset \mathds{D}^d\setminus\{\{0\}^d\}$ such that the first $d-1$ elements of $\lambda$ are the eigenvalues
of $A$, $\lambda_d=\lambda_{d-1}$ and $C_\varphi L_{d}=\lambda_{d-1}L_d+L_{d-1}$.
\begin{proposition}\label{eigenvector}
	Let $f\in\F(\C^d)$ be an eigenvector of $C_\varphi$. Then, $f_\alpha=0$ for any $\alpha\in \N^d$ such that $\alpha_d\neq 0$.  
\end{proposition} 

\begin{proof}	
	Let us proceed towards a contradiction. 
	Let $f\in \F(\C^d)$ be an eigenvector of $C_\varphi $ such that there is $\widehat{\gamma}\in \N^d$ with $\widehat{\gamma}_d\neq 0$ and $f_{\widehat{\gamma}}\neq 0$.
	Let us denote by $\Lambda\in \C$ the eigenvalue associated to $f$.
	Let $\gamma\in\N^d$ be such that $\gamma_j=\widehat{\gamma}_j$ for all $j=1,...,d-2$, $|\gamma|=|\widehat{\gamma}|$, $f_\gamma\neq 0$ and $\gamma_d$ is maximal. \\
	Now, notice that for any $z\in \C^d$
	\[C_\varphi f(z)= \sum_{n=0}^\infty\sum_{\substack{\alpha\in \N^d\\ |\alpha|=n}} f_\alpha L^\alpha(\varphi(z))=\sum_{n=0}^{+\infty} \sum_{\substack{\alpha\in \N^d\\ |\alpha|=n} } f_\alpha  \left(\lambda_{d-1} L_{d}(z) + L_{d-1}(z) \right)^{\alpha_d}\prod_{j=1}^{d-1} \lambda_j^{\alpha_j} L_j^{\alpha_j}(z) .\]
	Recalling that $\Lambda$ is the eigenvalue associated to $f$, for any $z\in \C^d$ we have that 
	\begin{align}\label {eq: prop eigenvector}
	C_\varphi f (z)= \Lambda  \sum_{n=0}^\infty\sum_{\substack{\alpha\in \N^d\\ |\alpha|=n}} f_\alpha L^\alpha(z).
	\end{align}
	Since for every function $g\in \F(\C^d)$ there is a unique sequence $(g_\alpha)_{\alpha\in \N^d}$ such that $g=\sum_{n=0}^{+\infty}\sum_{|\alpha|=n}g_\alpha L^\alpha$, the coefficients of both sides of \eqref{eq: prop eigenvector} coincide. 
	Therefore, regarding the coefficients that multiply $L^\gamma$ and $L^{\gamma+e(d-1)-e(d)}$ we get that
	
	\begin{align*}
	\lambda^\gamma f_\gamma  = &\Lambda f_\gamma\\
	\lambda^{\gamma-e(d)}\gamma_d f_\gamma+\lambda^\gamma f_{\gamma+e_{d-1}-e_d} = & \Lambda f_{\gamma+e_{d-1}-e_d}.
	\end{align*}
	Thus, since $f_\gamma\neq 0$, it follows that $\Lambda =\lambda^\gamma\neq 0$. However, since $\gamma_d\neq 0$ and
	$\lambda^{\gamma-e(d)}\neq 0$, the second equality gives us that $f_\gamma=0$ which is a contradiction.
	
\end{proof}
As a direct consequence of Proposition~\ref{eigenvector} we get:

\begin{proposition}\label{spectrum compact}
	The spectrum of $C_\varphi$ is $\sigma(C_\varphi)=\{\lambda^\alpha:~\alpha\in \N^{d-1}\times\{0\}\}\cup\{0\}$.
\end{proposition}

\begin{proof}
	Since $C_\varphi$ is a compact operator, we know that $\sigma(C_\varphi)= \sigma_p(C_\varphi)\cup\{0\}$.
It follows from the proof of Proposition~\ref{eigenvector} that the eigenvalues of $C_\varphi$ are of the form $\lambda^\alpha$, with $\alpha\in \N^{d-1}\times \{0\}$.
	Conversely, we have that $C_\varphi L^\alpha=\lambda^\alpha L^\alpha$ for all $\alpha\in \N^{d-1}\times \{0\}$.
\end{proof}
Observe that, in fact, we have shown that for any $\alpha\in \N^{d-1}\times \{0\}$, $\ker(C_\varphi-\lambda^\alpha Id)=\textup{span}\{L^\beta:~\beta\in\N^{d-1}\times\{0\},~\lambda^{\alpha-\beta}=1\}$.\\

Let us now fix an enumeration $(\beta(n))_n$ of $\N^{d-1}\times\{0\}$ such that the sequence $(|\lambda^{\beta(n)}|)_n$ is nonincreasing. 
Also, consider $(R(n))_n\subset \R$ a strictly decreasing enumeration of the set $\{|\lambda^{\beta(n)}|:~n\in\N \}$.
For any $n\in \N$ let us consider the set 
\[\II(n):=\{\alpha\in \N^d:~ \beta(n)_{d-1}=\alpha_{d-1}+\alpha_d,~\alpha_j=\beta(n)_j~\text{for }j=1,...,d-2\}\]
which is a finite subset of $\N^d$.
Let $N\in \N$ and denote by $Y_N$ the subspace of $\F(\C^d)$ defined by

\[Y_N:=\Big\{f\in \F(\C^d):~ f(z)=\sum_{\alpha\in \N^d\setminus \cup_{n=0}^N \II(n)}f_\alpha L^{\alpha}(z)\Big\}.\]
Also, let us denote by $\Xi_N$ the linear projection on $\F(\C^d)$ defined by

\[f\in \F(\C^d)\mapsto \Xi_N(f)(z):= \sum_{\alpha\in \N^d\setminus \cup_{n=0}^N \II(n)} f_\alpha L^{\alpha}(z),~\text{for all }z\in \C^d.\] 

In the following proposition we collect some facts related to $Y_N$ and $\Xi_N$.
\begin{proposition}\label{facts of xi and y}
	Let $N\in \N$. Then:
	\begin{enumerate}[(a)]
		\item $\Xi_N$ is a bounded projection onto $Y_N$. In particular, $Y_N$ is a closed subspace.
		\item $\Xi_N$ and $C_\varphi$ commute.
		\item $Y_N$ is an invariant subspace of $C_\varphi$.
		\item $\sigma(C_\varphi|_{Y_N})= \{\lambda^{\beta(n)}:~n>N\}\cup \{0\}$.
		
	\end{enumerate}
\end{proposition}

\begin{proof}
	(a): Observe that $I-\Xi_N$ is the linear operator defined by $(I-\Xi_N)f= \sum_{\alpha\in \bigcup_{n=0}^N \II(n)} f_\alpha L^{\alpha}$ for all $f\in \F(\C^d)$. 
	Therefore, since $\II(n)$ is a finite subset of $\N^d$ for each $n\in\N$, by Proposition~\ref{Projection and complementability} and Proposition~\ref{computing projection} we get that $I-\Xi_N$ is bounded. 
	Thus, $\Xi_N$ is bounded as well.\\
	(b): It easily follows from the fact that $I-\Xi_N$ and $C_\varphi$ commute.\\
	(c): It directly follows from (a) and (b).\\
	(d): By (c), $C_\varphi|_{Y_N}\in \mathcal{L}(Y_N)$ is a compact operator. Followed by some straightforward modifications, the argument presented in the proof of Proposition~\ref{spectrum compact} can be used to show that $\sigma(C_\varphi|_{Y_N})= \{\lambda^{\alpha(n)}:~n>N\}\cup \{0\}$.
\end{proof}

Now we are in shape to prove Theorem~\ref{theorem cyclic vectors}.
\begin{proof}[Proof of Theorem~\ref{theorem cyclic vectors}]
	Let us assume that $A$ is non-diagonalizable. 
	The case where $A$ is diagonalizable is simpler and the argument to prove this theorem follows the same line as the presented proof. 
	Since $A$ is compact and cyclic, by Theorem~\ref{Theorem 1}, we know that $\|A\|<1$, that $A$ is invertible and that the canonical Jordan form of $A$ admits a Jordan block of size exactly $2$.\\
	
	Let $f\in \F(\C^d)$ be a cyclic vector for $C_\varphi$. 
	Let us assume, towards a contradiction, that there is $\widehat{\alpha}\in \N^{d}$ such that $\widehat{\alpha}_{d-1}=0$ and $f_{\widehat{\alpha}}=0$.
	Observe that, for any $j\in \N$, we have that
	
	\begin{align*}
	C^j_\varphi f (z)&= \sum_{n=0}^\infty \sum_{\substack{ \alpha\in\N^d\\|\alpha|=n}} f_\alpha\lambda^{j\alpha}\left(L_d(z)+\dfrac{j}{\lambda_{d-1}} L_{d-1}(z)\right)^{\alpha_d} \prod_{k=1}^{d-1}L_k(z)^{\alpha_k} \\
	&=\sum_{n=0}^\infty \sum_{\substack{ \alpha\in\N^d\\|\alpha|=n}} L^\alpha \lambda^{j\alpha}\sum_{l=0}^{\alpha_{d-1}}\begin{pmatrix}
	\alpha_{d}+l\\ l
	\end{pmatrix}\dfrac{j^l}{\lambda_{d-1}^l}f_{\alpha-le(d-1)+le(d)}.
	\end{align*}
	
	Therefore, we have that $(C_\varphi^j f)_{\widehat{\alpha}}=0$ for all $j\in \N$. 
	This implies that the sequence $(P_{|\widehat{\alpha}|}(C_\varphi^j f))_j$ is contained in a subspace 
	of $\PP_{|\widehat{\alpha}|}$ of dimension $\textup{dim}(\PP_{|\widehat{\alpha}|})-1$.
	Thus, $f$ is not a cyclic vector.\\
	
	Conversely, let $f\in \F(\C^d)$ be such that $f_\alpha\neq 0$ for all $\alpha\in \N^d$ with $\alpha_{d-1}= 0$. 
	In order to prove that $f$ is a cyclic vector for $C_\varphi$ we follow an argument which is similar to the one given in the proof of Theorem~\ref{Theorem 1}.
	So here we only sketch the proof, highlighting the main differences with that of Theorem~\ref{Theorem 1}.\\
	
	For $n\in\N$, let us set
	$\mathcal{R}(n):=\{\alpha\in \N^d:~|\lambda^\alpha|=R(n)\}$. 
	Since $\lambda\subset \D^d$, the set $\mathcal{R}(n)$ is finite for any $n$ and $\{\mathcal{R}(n):~n\in\N\}$ is a partition of $\N^d$.
	Observe that $R(0)=1$ and $\mathcal{R}(0)=\{\{0\}^d\}$.
	Let $H=\overline{\textup{span}}\{C^j_\varphi f:~j\in\N\}$.
	In what follows, we prove that $H=\F(\C^d)$ by induction in the following way: at each step we show that $\{L^\alpha:~\alpha\in \mathcal{R}(n)\}\subset H$.\\
	
	\textbf{Initialization step.} Set $\mathcal{O}=(0,\dots,0)\in \N^d$. For any $j\in \N$, we have that
	\[C_\varphi^j f= f_{\mathcal{O}} + C_\varphi^j(f - f_{\mathcal{O}})= f_{\mathcal{O}} + C_\varphi^j(\Xi_0(f)).\]
	
	Notice that $\Xi_0(f)\in Y_0$ and $\sigma(C_\varphi|_{Y_0})\subset \D$. 
	Indeed, by Proposition~\ref{facts of xi and y}~(4), $\sigma(C_\varphi|_{Y_0})=\{\lambda^\alpha:~\alpha\in \N^{d-1}\times\{0\}\setminus\{\mathcal{O}\}\}$.
	Thanks to the spectral radius formula and since $\lambda\in \D$, the sequence $(\|C_\varphi|_{Y_0}^j\|)_j$ tends to $0$ as $j$ tends to infinity. 
	Hence, the sequence $(C_\varphi^j f)_j$ converges to $f_{\mathcal{O}} \in H$ and thus, since $f_{\mathcal{O}}\neq 0$, the constant functions belong to $H$.\\
	
	\textbf{Inductive step.} Let us assume that for some $n\geq 1$, $\{L^\alpha:~\alpha\in \mathcal{R}(k),~k\leq n-1\}\subset H$. 
	We prove that $\{L^\alpha:~\alpha\in \mathcal{R}(n)\}\subset H$. 
	Let us consider the sequence $(m(k))_k$ defined by $m(k)= \max\{j\in \N:~ \beta(j)\in \mathcal{R}(k)\}$.
	Observe that, thanks to the induction hypothesis, $\Xi_{m(n-1)}f\in H$.
	Also, notice that
	
	\[\Xi_{m(n-1)}f= \sum_{\alpha\in \mathcal{R}(n)}f_\alpha L^\alpha + \Xi_{m(n)}f.\]
	
	By Proposition~\ref{facts of xi and y}, $\Xi_{m(n)}f\in Y_{m(n)}$ and $\sigma(C_\varphi|_{Y_{m(n)}})\subset R(n+1)\overline{\D}$.
	Therefore, since $R(n+1)<R(n)$, there is $\varepsilon >0$ and $J\in \N$ such that 
	\[\left\|  C_\varphi|_{Y_{m(n)}}^j \right\|\leq (R(n)-\varepsilon)^j,~\text{for all } j\geq J. \]
	
	Thus, for any $\alpha\in \mathcal{R}(n)$, $\lambda^{-j\alpha} C_\varphi^j(\Xi_{m(n)}f)$ tends to $0$ as $j$ tends to infinity. 
	At this point, the proof follows closely the lines of the proof of Theorem~\ref{Theorem 1}. 
	Indeed, observe that
	
	\begin{align*}
	C_\varphi^j\Xi_{m(n-1)}f=& \sum_{\alpha\in \mathcal{R}(n)} L^\alpha \lambda^{j\alpha}\sum_{l=0}^{\alpha_{d-1}}\begin{pmatrix}
	\alpha_{d}+l\\ l
	\end{pmatrix}\dfrac{j^l}{\lambda_{d-1}^l}f_{\alpha-le(d-1)+le(d)}\\
	&+C_\varphi^j\Xi_{m(n)}f
	\end{align*}
	Thus, the coefficient associated to $L^\alpha$, with $\alpha\in \mathcal{R}(n)$, is $\lambda^{j\alpha}$ times a polynomial on $j$ of degree $\alpha_{d-1}$ due to the fact that $f_{\alpha -\alpha_{d-1}e(d-1)+\alpha_{d-1}e(d)}\neq 0$.
	Now, consider $\widehat{\alpha}\in \mathcal{R}(n)$ and $K=\max\{\alpha_{d-1}:~\alpha\in \mathcal{R}(n)\}$. We study inductively the sequences
	\[ \left(\dfrac{C_\varphi^j\Xi_{m(n-1)}f}{\lambda^{j\widehat{\alpha}}j^K}\right)_j,~ \left(\dfrac{C_\varphi^j\Xi_{m(n-1)}f-p_{j,1}}{\lambda^{j\widehat{\alpha}}j^{K-1}}\right)_j, ~\cdots ~, ~\left(\dfrac{C_\varphi^j\Xi_{m(n-1)}f-p_{j,K}}{\lambda^{j\widehat{\alpha}}}\right)_j,\]
	where $(p_{j,k})_{j,k}\subset \F(\C^d)$ are the polynomials defined as follows:
	\[p_{j,k}=\sum_{\substack{\alpha\in \mathcal{R}(n)\\ \alpha_{d-1}\geq K+1-k}} L^\alpha \lambda^{j\alpha}\sum_{l=0}^{\alpha_{d-1}}\begin{pmatrix}
	\alpha_{d}+l\\ l
	\end{pmatrix}\dfrac{j^l}{\lambda_{d-1}^l}f_{\alpha-le(d-1)+le(d)}.\]
	
	By sending $j$ to infinity on the first sequence and mimicking the proof of Theorem \ref{Theorem 1} (see in particular Lemma \ref{obtaining L alpha}), we obtain that each
	$L^\alpha$, with $\alpha\in \mathcal{R}(n)$, such that its associated coefficient is $\lambda^{j\alpha}$ times a polynomial of degree $K$ on $j$, belongs to $H$. 
	Observe that in the second sequence $p_{j,1}$ cancels all the $L^\alpha$ of $C_\varphi^j\sum_{\alpha\in \mathcal{R}(n)}f_\alpha L^\alpha$ such that their coefficient is $\lambda^{j\alpha}$ times a polynomial on $j$ of degree $K$.
	Thus, the second sequence is contained in $H$. 
	By sending $j$ to infinity on the second sequence, we obtain that each $L^\alpha$, with $\alpha\in \mathcal{R}(n)$, such that its associated coefficient is $\lambda^{j\alpha}$ times a polynomial of degree $K-1$ on $j$ belongs to $H$.
	The polynomial $p_{j,2}$ cancels all the $L^\alpha$ of $C_\varphi^j\sum_{\alpha\in \mathcal{R}(n)}f_\alpha L^\alpha$ such that the associated coefficient is $\lambda^{j\alpha}$ times a polynomial on $j$ of degree $K$ or $K-1$. This procedure leads to a finite induction which ends in $K+1$ steps obtaining that $\{L^\alpha:~\alpha\in \mathcal{R}(n)\}\subset H$.\\

	\textbf{Conclusion.} Since $\textup{span}\{z^\alpha:~\alpha\in \N^d\}=\textup{span}\{L^\alpha:~\alpha\in\N^d\}\subset H$ we obtain that $H = \F(\C^d)$. Thus, $f$ is a cyclic vector for $C_\varphi$.
\end{proof}

\begin{remark}
	If $A$ is diagonalizable, we actually have that
	\[C^j_\varphi f(z)=\sum_{n=0}^\infty \sum_{\substack{ \alpha\in\N^d\\|\alpha|=n}} f_\alpha \lambda^{j\alpha}L^\alpha(z),~ \text{for all } z\in \C^d.\]
\end{remark}

As an immediate corollary of Theorem \ref{theorem cyclic vectors} we can state the following result.
\begin{corollary}
Let $C_\varphi$ be a compact cyclic composition operator on $\mathcal F(\C^d)$ and denote by $\textrm{Cyc}(C_\varphi)$ its set of cyclic vectors. Then $\textrm{Cyc}(C_\varphi)\cup\{0\}$  does not contain a subspace of dimension $2$.  
\end{corollary}

\section{Further dynamical properties of composition operators}\label{section 6}

This section is devoted to prove that composition operators on $\F(\C^d)$ are never weakly-supercyclic nor convex-cyclic.\\

A bounded linear operator $T$ defined on a separable Banach space $X$ is said supercyclic with respect to the topology $\tau$ if $\C\cdot \textup{orb}(T,x)$ is dense in $(X,\tau)$. 
In \cite[Theorem 5.4]{JPZ} it is proven that composition operators defined on $\F(\C^d)$ are never supercyclic. 
Also, in \cite[Theorem 1.7]{M} it is proven that weighted composition operators defined on $\F(\C)$ are never supercyclic with respect to the pointwise convergence topology. 
The proof of our next result is an adaptation of that of \cite[Theorem 1.7]{M}.
Regardless, we provide it for the sake of completeness.

\begin{theorem}\label{Theorem tp supercyclic}
	Let $\varphi(z):=Az+b$ be a holomorphic map such that $C_\varphi$ induces a bounded composition operator on $\F(\C^d)$. 
	Then, $C_\varphi$ is not supercyclic with respect to the pointwise convergence topology.
	In particular, there is no weakly-supercyclic composition operator defined on $\F(\C^d)$.	
\end{theorem}

\begin{proof}[Proof of Theorem~\ref{Theorem tp supercyclic}]

	Let $\xi\in \C^d$ be a fixed point of $\varphi$. 
	Thus, $\varphi(z)= A(z-\xi)+\xi$ for all $z\in \C^d$. 
	Let us proceed towards a contradiction. 
	Assume that there is $f\in \F(\C^d)$ such that $f$ is a supercyclic vector for $C_\varphi$ with respect to the pointwise convergence topology.
	It easily follows that $f(\xi)\neq 0$. 
	Thus, by \cite[Proposition 4]{BJM}, we have that for any $z,z'\in \C^d$, with $z\neq z'$, 
	\begin{equation}\label{eq:supercyclic}
	\overline{ \left\{  \dfrac{f(\varphi^n(z))}{f(\varphi^n(z'))}:~n\in\N,~ f(\varphi^n(z'))\neq 0 \right\}}=\C.
	\end{equation}
	
	\noindent Since $f(\xi)\neq0$, there is $r>0$ such that $0 \notin \overline{f(\xi+r\D)}$. 
	Also, since $\|A\|\leq 1$, we have that $\varphi(\xi+r\D)\subset \xi+r\D$. 
	Now, let us fix $z\in (r\D+\xi)\setminus \{\xi\}$ and set $z'=\xi$. 
	Then
	
	\[ \left | \dfrac{ f(\varphi^n(z))}{f(\varphi^n(z'))} \right | \leq  \dfrac{ \sup \{|f(w)|: ~w\in \xi+r\D \}}{|f(\xi)|}<\infty.  \]
	
	This clearly contradicts \eqref{eq:supercyclic}.

\end{proof}

Now, we turn our study to the concept of convex-cyclicity. A linear operator $T\in \mathcal{L}(X)$ is said convex-cyclic if there is $x\in X$ such that $\textup{co} (\textup{orb}(T,x))$ is dense in $X$, where $\textup{co}(A)$ means the convex hull of the set $A$.
Up to the best of our knowledge, this concept was introduced in 2013 by Rezaei \cite{R}.
Further, in \cite{M}, Mengestie characterized the convex-cyclicity of weighted composition operators defined on $\F(\C)$. 
In what follows, we show that there is no convex-cyclic composition operators on $\F(\C^d)$.

\begin{theorem}
	Let $\varphi:\C^d\to\C^d$ be a holomorphic map such that $C_\varphi$ induces a bounded composition operator on $\F(\C^d)$. Then, $C_\varphi$ is not convex-cyclic.
\end{theorem}

\begin{proof} 
Let $\xi\in \C^d$ be a fixed point of $\varphi$ and let $f\in \F(\C^d)$.
According to Proposition~\ref{Projection and complementability}, $\F(\C^d)=\PP_0\oplus \QQ_0$. 
Moreover, $\PP_0$ and $\QQ_0$ are invariant subspaces for $C_\varphi$ and $P_0$ and $C_\varphi$ commute. 
Let us denote by $f_0:=P_0(f)$, which is a constant function.  
Now, observe that, for any sequence $(\sigma_k)_{k\in \N}\subset\R^+$ with finitely many non-zero terms, such that $\sum_k \sigma_k=1$, we have that

\[P_0\left(\sum_{k=0}^\infty \sigma_k C^k_\varphi f\right)=\sum_{k=0}^\infty \sigma_k C^k_\varphi P_0 f = f_0.\]
Therefore, $f$ is not a convex-cyclic vector for $C_\varphi$. Since $f$ is an arbitrary function on $\F(\C^d)$, the operator $C_\varphi$ is not convex-cyclic.
\end{proof}

\section{Approximation numbers}
\label{section 7}
\noindent In order to simplify the notations, in this section we use the convention $0^0=1$.\\

\noindent Let us recall that for a linear bounded compact operator $T\in\mathcal{L}(X)$, the $n$-th approximation number $a_n(T)$ is defined by

\[ a_n(T):= \inf \{ \|T-S\|:~ S\in \mathcal{L}(X),~\textup{dim}(S(X))\leq n-1\}.\]

When $X$ is a separable Hilbert space, it is well-known that the sequence $(a_n(T))_n$ coincides with the decreasing enumeration of  the singular values of $T$, that is, the square roots of the eigenvalues of $T^*T$ (equivalently, the eigenvalues of $\sqrt{T^*T}$).

The computation (or, at least, the estimation) of the approximation numbers of compact composition operators has been the subject of many investigations in the recent years (see for instance \cite{LQR} and the references therein). We compute these numbers for compact composition operators defined on $\F(\C^d)$. To do this we gather some results which can be found in \cite[Theorem 1.2]{Z} and \cite[Proposition 2.5]{FZ}: 

\begin{proposition}\cite{FZ,Z}\label{weighted operators}
	Let $\varphi(z):= Az+b$ be such that $C_\varphi$ induces a bounded compact composition operator on $\F(\C^d)$, i.e. $\|A\|<1$. 
	Further, assume that $A$ is self-adjoint. 
	Then, $W_{k_b,\varphi}$ is a bounded self-adjoint weighted composition operator on $\F(\C^d)$. 
	Moreover, the operator $W_{k_b,\varphi}^*W_{k_b,\varphi}$ is unitary equivalent to $\exp(\langle (I-A)^{-1}b,b\rangle) C_{AA^*z}$. 
\end{proposition}

In \cite[Theorem 1.1]{FZ} we can find the spectrum of bounded normal weighted composition operators defined on $\F(\C^d)$. A straightforward modification of the proof
of Proposition \ref{spectrum compact} gives us the multiplicity of each non zero eigenvalue.
\begin{lemma}\label{eigenvalues}
	Let $A\in \C^{d\times d}$ be an Hermitian matrix of norm $\|A\|<1$. Let $\lambda:=(\lambda_j)_{j=1}^d$ be the eigenvalues of $A$. Then, 
	\[\sigma (C_{Az})= \{\lambda^\alpha:~ \alpha\in \N^d\}\cup\{0\}.\]
	Moreover, the multiplicity of the eigenvalue $\rho\in \sigma_p(C_{Az})\setminus\{0\}$ is exactly $\#\{\alpha\in \N^d:~\rho=\lambda^\alpha\}$.
\end{lemma}

\begin{theorem}
Let $\varphi(z):=Az+b$ be such that $C_\varphi$ induces a bounded compact composition operator on $\F(\C^d)$ with $A\neq 0$.
Let $\lambda=(\lambda_j)_{j=1}^d\subset \R^+$ be the singular values of $A$.
Let $(\alpha_n)_n\subset \N^d$ be an enumeration of the set $\{\alpha\in\N^d:~\lambda^\alpha\neq 0\}$ such that the sequence $(\lambda^{\alpha_n})_n$ is nonincreasing. 
Then

\[a_n(C_\varphi)= \exp\left(\dfrac{\langle (I-B)^{-1}v,v\rangle}{2}- \frac{|v|^2}{4}\right) \lambda^{\alpha_n},\]
where $B=\sqrt{AA^*}$ and $v=(I+B)^{-1}b$. 
In particular, 
\[\sum_{n=1}^\infty a_n(C_\varphi)=\exp\left(\dfrac{\langle (I-B)^{-1}v,v\rangle}{2}- \frac{|v|^2}{4}\right) \prod_{j=1}^{d}\dfrac{1}{1-\lambda_j}.\]
\end{theorem}

\begin{proof}
Let us first notice that, since $C_\varphi$ is compact, $\|A\|<1$.
Thanks to Proposition~\ref{adjoint operator}, we know that $C_\varphi^*=W_{k_b,A^*z}$. 
Thus
\[T:=C_\varphi^*C_\varphi = W_{k_b,AA^*z+b}.\]

In order to continue, set $B:= \sqrt{AA^*}$ and recall that $B$ is a self-adjoint matrix.  
Observe that $\|B\|<1$ as well.
Let $v:= (I+B)^{-1}b$ and define 

\[S = \exp\left(-\dfrac{|v|^2}{4} \right) W_{k_v,Bz+v}.\]

We claim that $S:= \sqrt{T}$. 
Indeed, thanks to Proposition~\ref{weighted operators}, $S$ is a bounded self-adjoint weighted composition operator on $\F(\C^d)$.
Moreover, observe that for any $f\in \F(\C^d)$ we have
\begin{align*}
S^2 (f)(z)&= \exp\left(-\dfrac{|v|^2}{2} \right)k_v(z) k_v(Bz+v)f(B^2z+Bv+v)\\
&=\exp\left(\dfrac{-|v|^2 + \langle z,v\rangle + \langle Bz+v,v \rangle }{2} \right)f(AA^*z+(I+B)v)\\
&=\exp \left( \dfrac{\langle z,b \rangle }{2}\right)f(AA^*z+b)= T(f)(z).
\end{align*} 

Since $S$ is a self-adjoint operator, $T=S^*S$.
Again thanks to Proposition~\ref{weighted operators}, $T$ is unitarily equivalent to $\exp({\langle (I-B)^{-1}v,v\rangle}-\frac{|v|^2}{2}) C_{B^2z}$.
Now, thanks to Lemma~\ref{eigenvalues} and recalling that $B^2=AA^*$, we get that the eigenvalues of $T$ are

\[\sigma(T):= \left \{ \exp\left({\langle (I-B)^{-1}v,v\rangle}-\frac{|v|}{2}\right) (\lambda^{\alpha})^2:~ \alpha\in \N^d  \right \}.\]

Moreover, if $(\alpha_n)_n\in \N^d$ is an enumeration of the set $\{\alpha\in\N^d:~\lambda^\alpha\neq 0\}$ such that the sequence $(\lambda^{\alpha_n})_n$ is nonincreasing, then 

\[a_n(C_\varphi)= \left( \exp\left({\langle (I-B)^{-1}v,v\rangle} - \frac{|v|^2}{2}\right)\lambda^{2\alpha_n} \right)^{1/2}= \exp\left(\dfrac{\langle (I-B)^{-1}v,v\rangle}{2}-\frac{|v|^2}{4}\right)\lambda^{\alpha_n}.\]

Finally, the formula of geometric series gives us that

\[\sum_{n=1}^\infty a_n (C_\varphi)= \exp\left(\dfrac{\langle (I-B)^{-1}v,v\rangle}{2}-\frac{|v|^2}{4}\right) \prod_{j=1}^d \dfrac{1}{1-\lambda_j}.\]
\end{proof}

\begin{remark}
With this theorem, we get another proof that a compact composition operator 
on the Fock space belongs to all Schatten classes (see \cite{Du} or \cite{JPZ}).
\end{remark}

\bibliographystyle{amsplain}

\providecommand{\bysame}{\leavevmode\hbox to3em{\hrulefill}\thinspace}
\providecommand{\MR}{\relax\ifhmode\unskip\space\fi MR }
\providecommand{\MRhref}[2]{%
  \href{http://www.ams.org/mathscinet-getitem?mr=#1}{#2}
}
\providecommand{\href}[2]{#2}
\begin{thebibliography}{}

\end{thebibliography}


\begin{thebibliography}{999}
\bibitem{BM} F. Bayart and \'E. Matheron. Dynamics of linear operators. Cambridge Tracts in Mathematics, 179.
\bibitem{BJM} M. J. Beltr\'an-Meneu, E. Jord\'a and M. Murillo-Arcila. Supercyclicity of weighted composition operators on spaces of continuous functions, Collect. Math., 71 (2020), 493-509.
\bibitem{BS} P. Bourdon and J. Shapiro. Cyclic phenomena for composition operators. Mem. Amer. Math. Soc, 125 (1997), No. 596.
\bibitem{CG} T. Carrol and C. Gilmore. Weighted composition operators on Fock spaces and their dynamics. Preprint: arXiv:1911.07254
\bibitem{CMS} B. Carswell, B. MacCluer and A. Schuster. Composition operators on the Fock space. Acta Sci. Math. (Szeged) 69 (2003) 871-887.
\bibitem{CM} C. Cowen and B. MacCluer. Composition operators on spaces of analytic functions. Studies in Advanced Mathematics. CRC Press.
\bibitem{DK} M. Doan and L. Khoi. Closed range and cyclicity of composition operators on Hilbert spaces of entire functions. Complex Var. Elliptic Equ., 63(11) (2018), 1558-1569.
\bibitem{Du} D.Y. Du. Schatten Class Weighted Composition Operators
on the Fock Space $F^2_\alpha(\mathbb C^N )$. Int. Journal of Math. Analysis, Vol. 5, 2011, no. 13, 625 - 630
\bibitem{FZ} L. Feng and L. Zhao. Spectrum of normal weighted composition operators on the Fock space Over $\C^N$. Acta Math. Sin., Eng. Ser., 35(9) (2019), 1563–1572. 
\bibitem{GI} K. Guo and K. Izuchi. Composition operators on Fock type spaces. Acta Sci. Math. (Szeged) 74 (2008) 807-828.
\bibitem{JPZ} L. Jiang, G. Prajitura and R. Zhao. Some characterizations for composition operators on the Fock spaces.
J. Math. Anal. Appl. 455 (2017) 1204-1220.
\bibitem{LQR} D. Li, H. Queffélec and L. Rodríguez-Piazza. Approximation and entropy numbers of composition operators, Concrete Operators 7 (2020) 166--179.
\bibitem{M} T. Mengestie. Convex-cyclic weighted composition operators and their adjoints. Preprint: arXiv:2112.05371v1
\bibitem{M2} T. Mengestie. Cyclic and supercyclic weighted composition operators on the Fock space. Preprint: arXiv:1901.01697v1
\bibitem{M3}
T. Mengestie.
Dynamics of weighted composition operators and their adjoints on the Fock space.
Complex Anal. Oper. Theory 16 (2022), 27.
\bibitem{R} H. Rezaei. On the convex hull generated by orbit of operators. Linear Algebra Appl., 438(11) (2013), 4190-4203.
\bibitem{U} S. Ukei. Weighted composition operator on the Fock space. Proc. Amer. Math. Soc., 135(5) (2007), 1405-1410.
\bibitem{ZZ} L. Zhang and Z. Zhou. Hypercyclicity of weighted composition operators on a weighted Dirichlet space. Complex Var. Elliptic Equ., 59(7) (2014), 1043-1051, 
\bibitem{Z2} L. Zhao. Invertible weighted composition operators on the Fock space of $\C^N$. Journal of function spaces, (2015).
\bibitem{Z} L. Zhao. Normal weighted composition operators on the Fock space of $\C^n$.  Oper. Matrices, 11(3) (2017), 697-704.
\end{thebibliography}

\end{document}